\documentclass[a4paper,reqno]{amsart}

\usepackage[USenglish]{babel}
\usepackage{amssymb,amscd}
\usepackage{amsxtra}
\usepackage{amsfonts} 
\usepackage{amsthm}
\usepackage{nicefrac}
\usepackage{enumerate}
\usepackage{psfrag}
\usepackage{diagrams}
\usepackage{enumerate}
\usepackage{microtype}
\usepackage{todonotes}
\usepackage{xspace}
\usepackage[all]{xy}
\usepackage[pdftitle={Cutting out arithmetic Teichmueller curves in genus two
via Theta functions},
	    pdfauthor={Andre Kappes and Martin Moeller},
	    linktocpage=true]{hyperref}

\newlength{\halfbls}

\newtheorem{theorem}{Theorem}[section]
\newtheorem{prop}[theorem]{Proposition} 
\newtheorem{cor}[theorem]{Corollary}
\newtheorem{rem}[theorem]{Remark}

\newtheorem{lemma}[theorem]{Lemma}

\newtheorem{defn}[theorem]{Definition}

\theoremstyle{remark}

\renewcommand{\div}{\mathrm{div}}

\newcommand{\PP}{\mathbb{P}}

\newcommand{\HH}{\mathbb{H}}

\newcommand{\QQ}{\mathbb{Q}}

\newcommand{\RR}{\mathbb{R}}
\newcommand{\CC}{\mathbb{C}}
\newcommand{\ZZ}{\mathbb{Z}}
\newcommand{\NN}{\mathbb{N}}

\newcommand{\Sp}{\mathrm{Sp}}
\newcommand{\SL}{\mathrm{SL}}

\newcommand{\CH}{\mathrm{CH}}

\newcommand{\PSL}{\mathrm{PSL}}

\newcommand{\Stab}{\mathrm{Stab}}
\newcommand{\End}{\mathrm{End}}

\newcommand{\cA}{\mathcal{A}}

\newcommand{\cO}{\mathcal{O}}

\newcommand{\moduli}[1][g]{{\mathcal M}_{#1}}
\newcommand{\omoduli}[1][g]{{\Omega \mathcal M}_{#1}}

\newcommand{\tr}{{\rm tr}}

\newcommand{\Per}{\mathrm{Per}}

\newcommand{\divides}{\mid}

\DeclareMathOperator{\Spec}{\mathrm{Spec}}

\newcommand{\ol}[1]{{\overline{#1}}}

\newcommand{\wt}[1]{{\widetilde{#1}}}

\renewcommand{\tilde}[1]{\widetilde{#1}}

\renewcommand{\Im}{\mathrm{Im\,}}       

\newcommand{\dual}{\vee} 
\newcommand{\isom}{\cong}

\newcommand{\pr}{\mathit{pr}}
\newcommand{\red}{\mathit{red}}
\newcommand{\olred}{\overline{\mathit{red}}}

\def\wt#1{\widetilde{#1}}

\newcommand{\e}{\mathbf{e}}

\newcommand{\barop}[2]{\big|_{#1} \bigl[#2\bigr]}

\DeclareMathOperator{\Aut}{Aut}

\DeclareMathOperator{\Jac}{Jac}

\DeclareMathOperator{\trace}{Tr}

\DeclareMathOperator{\diag}{diag}

\DeclareMathOperator{\Pic}{Pic}

\newcommand{\spin}{\varepsilon}
\newcommand{\tors}{\mathit{M}}
\newcommand{\ttors}{\mathrm{-tor}}
\DeclareMathOperator{\congruent}{\equiv}

\DeclareMathOperator{\semidirect}{\ltimes}
  
\def\be{\begin{equation}}   \def\ee{\end{equation}}     
\def\bes{\begin{equation*}}    \def\ees{\end{equation*}}
\def\ba{\be\begin{aligned}} \def\ea{\end{aligned}\ee}   
\def\bas{\bes\begin{aligned}}  \def\eas{\end{aligned}\ees}

\newcommand{\Teichmuller}{Teichm\"uller\xspace}

\newcommand{\frakod}{\mathfrak{o}_{d^2}}
\newcommand{\frako}{\mathfrak{o}}

\newcommand{\frakodual}{\mathfrak{o}_{d^2}^\dual}

\newcommand{\oodual}{\frako_{d^2}^\dual\oplus\frako_{d^2}}

\newcommand{\HMSoodual}{X_{d^2}}	
\newcommand{\oHMSoodual}{X_{d^2}^\circ}	

\newcommand{\Modulielldleveld}{X(d)_d}	
\newcommand{\oModulielldleveld}{X(d)_d^\circ}	

\newcommand{\Familyoodual}{{A}_{d^2}}  
\newcommand{\oFamilyoodual}{{A}_{d^2}^\circ} 

\newcommand{\oFamilyelld}{{E}[d]^\circ}

\newcommand{\Familyelldleveld}{{E}(d)_d} 
\newcommand{\oFamilyelldleveld}{{E}(d)_d^\circ} 

\newcommand{\SLoodual}{\Gamma_{d^2}}
\newcommand{\Semidirectoodual}{\tilde \Gamma_{d^2}}

\newcommand{\Jacforms}{\mathcal{J}}

\DeclareMathOperator{\tensor}{\otimes}   
\DeclareMathOperator{\Chern}{\mathrm{c}} 

\newcommand{\Tchi}[2]{{\bigl[ \begin{smallmatrix}  #1 \\ #2 \\ \end{smallmatrix} \bigr]}}

\newcommand{\sfrac}{\nicefrac}

\newcommand{\tvector}[2]{\left(\begin{smallmatrix} #1 \\ #2 \end{smallmatrix}\right)}
\newcommand{\tmatrix}[4]{\left(\begin{smallmatrix} #1&#2 \\ #3&#4 \end{smallmatrix}\right)}
\newcommand{\skp}[2]{\bigl\langle #1, #2 \bigr\rangle}		
\newcommand{\pder}[2][]{\frac{\partial #1}{\partial #2}}  			
\newcommand{\pdern}[3][]{\frac{\partial^{#3} #1}{\partial #2^{#3}}}  
\newcommand{\set}[2]{\bigl\{#1\mid #2\bigr\}} 		
\newcommand{\ie}{i.\,e.\ }

\makeatletter
\renewcommand{\subsubsection}{\@startsection{subsubsection}{2}%
        {\z@}{-3.25ex plus -1ex minus-.2ex}{-1em}{\bf}}
\makeatother
\def\={\;=\;}  \def\+{\,+\,} \def\-{\,-\,}
\def\be{\begin{equation}}   \def\ee{\end{equation}}     \def\bes{\begin{equation*}}    \def\ees{\end{equation*}}
\def\ba{\be\begin{aligned}} \def\ea{\end{aligned}\ee}   \def\bas{\bes\begin{aligned}}  \def\eas{\end{aligned}\ees}

\title[Cutting out arithmetic Teichm\"uller curves]
{Cutting out arithmetic Teichm\"uller curves in genus two via Theta functions}

\author{Andr\'e Kappes and Martin M\"{o}ller}
\date{\today}

\thanks{The authors are partially supported
by the ERC-StG 257137. }

\address{Institut f\"{u}r Mathematik, Goethe-Universit\"{a}t Frankfurt, 
Robert-Mayer-Str. 6-8, 60325 Frankfurt am Main, Germany}

\email{kappes@math.uni-frankfurt.de}
\email{moeller@math.uni-frankfurt.de}

\begin{document} 

\begin{abstract}
We compute the class of arithmetic genus two Teichm\"uller curves in 
the Picard group of pseudo-Hilbert modular surfaces, distinguished
according to their torsion order and spin invariant. As an application, 
we compute the number of genus two square-tiled surfaces with these 
invariants.
\par
The main technical tool is the computation of divisor classes of
Hilbert Jacobi forms on the universal abelian surface over the 
 pseudo-Hilbert modular surface.
\end{abstract}

\maketitle
\setcounter{tocdepth}{1}

\tableofcontents

\section{Introduction}

The aim of this paper is to contribute to the classification of
arithmetic \Teichmuller curves  and the computation of their basic invariants.
The extension of the bundle of Jacobi forms to 
the universal family of abelian surfaces over pseudo-Hilbert modular surfaces
and the computation of its class will be our main technical tool.
\medskip
\paragraph{\textbf{Arithmetic \Teichmuller curves.}} Square-tiled surfaces are
covers of the square torus, ramified over at most one
point. Affinely deforming the squares into parallelograms yields a curve
in the moduli space of curves, called arithmetic \Teichmuller curve.
Non-arithmetic \Teichmuller curves, which are generated by flat surfaces that do
not arise via branched coverings of the torus, have been classified in genus two
(\cite{mcmullenspin}, \cite{mcmullentor}),
and in higher genus there is a growing number of partial results.
For \Teichmuller curves generated by square-tiled surfaces, the classification
problem is solved only for genus two surfaces with a single ramification
point (\cite{hubertlelievre} for prime degree coverings and \cite{mcmullenspin}
in general). They are classified by two invariants, the number of squares and
the spin.
\medskip
\paragraph{\textbf{Genus two, two ramification points.}} Genus two square-tiled
covers with two ramification points come with three obvious invariants. 
One is the spin invariant, the number of integral Weierstra\ss\
points.
The other two are the torsion order of the two branch points in a minimal
intermediate torus covering and the degree of this covering (see
Section~\ref{sec:origamis}). 
It is conjectured (and well-supported by computer experiments of Delecroix and
Leli\`evre) that these are the only invariants, i.e.\ 
that the set $T_{d,M,\spin}$ of genus two degree $d$ covers of the torus
with torsion order $M$ and spin~$\spin$ is irreducible. 
For one ramification point, both \cite{hubertlelievre} and 
\cite{mcmullenspin} solved the irreducibility
question combinatorially by exhibiting prototypes for the flat
surfaces and connecting any two of the same invariants by a change of
direction. 
This approach might work for two ramification points as well, but the
combinatorial complexity is challenging.
\par
This paper does not contain any picture of a flat surface. Instead we
propose to tackle the classification problem by first computing the class of
$T_{d,M,\spin}$ in the (rational) Picard group of a pseudo-Hilbert modular
surface
and in the second step to argue that this class is not too divisible 
and that potential summands cannot be \Teichmuller curves.
\medskip
\paragraph{\textbf{Counting square-tiled surfaces.}} In this paper, we complete
the first step in this program for odd $d$. As a result, 
we can solve the following counting problem. For $M=1$ this has been 
conjectured by Zmiaikou ({\cite[p.~67]{ZmiaikouThesis}}). 
\par 
\begin{theorem} \label{thm:intro_count11}
The number $t_{d,\tors, \spin}$ of reduced square-tiled
surfaces of genus two, two ramification points, odd degree $d$, torsion order
$\tors$ and spin invariant $\spin$ is given as follows.\\  
If $M>1$ is odd, then
\begin{align*} t_{d,\tors, \spin = 3} &\= \frac{1}{24}(d-1)  
  \Delta_d\frac{\Delta_\tors}{\tors} \qquad 
&t_{d,\tors, \spin = 1} &\= \frac{1}{8}(d-1) 
\Delta_d \frac{\Delta_\tors}{\tors}.
\end{align*}
If $M$ is even,  then there is no spin invariant and 
\[
t_{d,\tors} \= \frac{1}{6}(d-1)\Delta_d\frac{\Delta_\tors}{\tors}.
\]
If $M=1$, then 
\[
t_{d,\tors =1, \spin = 3} \= \frac{1}{24}(d-3)(d-5)\frac1d \Delta_d
\quad \text{and} \quad
t_{d,\tors =1, \spin = 1} \= \frac{1}{8}(d-1)(d-3)\frac1d \Delta_d. 
\]
\end{theorem}
\par
\begin{rem}
In principle, the same program can be carried out for even $d$, but it requires
performing similar computations as we present them for covering surfaces
with an extra level of two (see Section~\ref{sec:modifeven}). The
conjectural values for the counting problem are as follows. For $M>1$ and $d$
even we have
\begin{align*}
t_{d,\tors, \spin = 0} &\= \frac{1}{24}(d-1)  
\Delta_d \frac{\Delta_\tors}{\tors}\qquad 
&t_{d,\tors, \spin = 2} &\= \frac{1}{8}(d-1) 
\Delta_d \frac{\Delta_\tors}{\tors},
\end{align*}
and for $M=1$ and $d$ is even the values are
\[
t_{d,\tors =1, \spin = 0} \= \frac{1}{24}(d-2) \Delta_d 
\quad \text{and} \quad
t_{d,\tors =1, \spin = 2} \= \frac{1}{8}(d-2)(d-4)\frac1d \Delta_d\,. 
\]
\end{rem}
\par
The {\em sum} of contributions of the two spin structures
appeared in {\cite[Theorem~3]{Ka06} and in \cite{ems}}, see 
Proposition~\ref{prop:kani_EMS} for the conversion of the two methods 
of counting.
\medskip
\paragraph{\textbf{Classes in the Picard group}}
The above counting result is a consequence of the following statement that
gives the class of the (union of) \Teichmuller curves $T_{d,\tors,\spin}$
generated by the square-tiled surfaces of degree $d$, with 
torsion order $\tors$ and spin $\spin$ on the compactified pseudo-Hilbert
modular surface $\HMSoodual$, whose open part $\oHMSoodual$ parametrizes
abelian surfaces with multiplication by a pseudo-quadratic order. See
Section~\ref{sec:pseudoquadratic orders} for the definition of 
$X_{d^2}$ and the Hodge bundles $\lambda_i$.
\par
\begin{theorem} \label{thm:intro_class11}
Let $d$ be odd. The class of $T_{d,\tors,\spin}$ in $\Pic_\QQ(\HMSoodual)$
is given as follows. If $M>1$ is odd, then 
\ba \ 
 [T_{d,\tors,\spin=3}] &\=
\tfrac12 \,d\, \frac{\Delta_\tors}{\tors} \,\Bigl((1 - \tfrac{1}{d})\lambda_1 + (2 - \tfrac{2}{d})\lambda_2\Bigr), \\
 [T_{d,\tors,\spin=1}] &\=
\tfrac32 \, d\,\frac{\Delta_\tors}{\tors}\, \Bigl((1 - \tfrac{1}{d})\lambda_1 + (2 - \tfrac{2}{d})\lambda_2\Bigr).
\ea
If $M$ is even then 
\ba\ [T_{d,\tors}] \=
2 d\frac{\Delta_\tors}{\tors} \Bigl((1 - \tfrac{1}{d})\lambda_1 + (2 - \tfrac{2}{d})\lambda_2\Bigr).
\ea
If $M=1$, then 
\ba\ [T_{d,\tors=1,\spin=3}] &\=
\tfrac{1}{d} \Bigl(\tfrac12 (d-3)(d-5))\lambda_1 + (d-3)(d-5)\lambda_2\Bigr), \\
[T_{d,\tors=1,\spin=1}] &\=
\tfrac{3}{d}\Bigl((\tfrac12 (d-1)(d-3))\lambda_1 + (d-1)(d-3)\lambda_2\Bigr).
\ea
\end{theorem}
\par
\begin{rem}
The conjectural classes for the case $d$ even are given as follows.
If $M>1$ is odd, then
\ba \ [T_{d,\tors,\spin=0}] &\=
\tfrac12 d \frac{\Delta_\tors}{\tors} \Bigl((1 - \tfrac{1}{d})\lambda_1 + (2 -
\tfrac{2}{d})\lambda_2\Bigr), \\
[T_{d,\tors,\spin=2}] &\=
\tfrac32 d\frac{\Delta_\tors}{\tors} \Bigl((1 - \tfrac{1}{d})\lambda_1 + (2 -
\tfrac{2}{d})\lambda_2\Bigr). \\
\ea
If $M$ is even then 
\ba\ [T_{d,\tors}] \=
2 d\frac{\Delta_\tors}{\tors} \Bigl((1 - \tfrac{1}{d})\lambda_1 + (2 -
\tfrac{2}{d})\lambda_2\Bigr).
\ea
If $M=1$, then 
\ba\ [T_{d,\tors=1,\spin=0}] &\=
 \Bigl(\tfrac12(d-2)\lambda_1 + (d-2)\lambda_2\Bigr), \\
[T_{d,\tors=1,\spin=2}] &\=
\tfrac{3}{d} \Bigl(\tfrac12 (d-2)(d-4) \lambda_1 + (d-2)(d-4)\lambda_2\Bigr).
\ea
\end{rem}
\par
\medskip
\paragraph{\bf Strategy of the proof.} Instead of locating a \Teichmuller 
curve inside $X_{d^2}$, we locate the branch points of the covering map from
the flat surface to the torus inside the universal family $\oFamilyoodual$ of
abelian surfaces over the open subset $\oHMSoodual$.
The points that we want to single out lie on image of the flat surface
in its Jacobian (i.e.\ on the theta divisor), they are branch points 
(i.e.\ the derivative of the theta function vanishes in some direction), and they
have the property that their image in a certain intermediate elliptic curve is
$M$-torsion.
Theorem~\ref{thm:classes-H11-tors-M} expresses that the image of this
intersection of three divisorial conditions in $\HMSoodual$ is the \Teichmuller
curve. 
The basic idea to use theta functions builds on that in \cite{moelprym}, 
but there one could work entirely in the two-dimensional base, while most
of the difficulties here come from performing the triple intersection
in the four-dimensional total space. Of course, for intersection theory
calculations, we need to work on a reasonable (normal, at most quotient
singularities) compactification  $\Familyoodual$ of $\oFamilyoodual$. We recall
the background on toroidal compactifications and construct  $\Familyoodual$ in
Section~\ref{sec:toroidal}.
The family $\Familyoodual$ comes with some obvious divisors (boundary components, 
Hodge bundle, zero sections), whose intersection product is readily computed.
The goal is hence to express the ingredients of the triple intersection in 
these terms.
\par
\medskip
\paragraph{\bf Jacobi forms for pseudo-Hilbert modular surfaces.}
Hilbert Jacobi forms are functions on the universal covering $\HH^2 \times
\CC^2$ of $\oFamilyoodual$ whose transformation law combines the elliptic
behavior 
on $\CC^2$ and the modular behavior on $\HH^2$ in the usual way as
for elliptic Jacobi forms. The precise definitions are given in
Section~\ref{sec:HJforms}. 
The basic example of a Jacobi form is the theta function, both in the 
elliptic and in the pseudo-Hilbert modular case. We would like to express
the divisor class of a Jacobi form on $\Familyoodual$ in terms of the
natural divisors mentioned above. We stress that, however, this question
is {\em not even well-defined}. Only after making some artificial choice
at the boundary (our choice is~\eqref{eqn:HJF_bundel} in Section~\ref{sec:HJforms}) 
we can determine the class of a Jacobi form in Theorem~\ref{thm:class_of_HJF}.
\par
At the end of the day, we are only interested in the class of a divisor
(the \Teichmuller curve) generically lying in $X^\circ_{d^2}$. Consequently, 
we have to determine and subtract in Section~\ref{sec:bdcontrib} the spurious
boundary
components, thereby compensating the arbitrariness in the boundary 
extension of Jacobi forms.
\par
Finally, in the case of $M=1$, the analogous statement of 
Theorem~\ref{thm:classes-H11-tors-M} is Theorem~\ref{thm:classes-H11-tors-zero}
and there two other spurious summands occur. One contribution is from 
the reducible locus in $X_{d^2}$, whose class we determine in
Section~\ref{sec:redlocus}.
The other contribution stems from square-tiled surfaces with only one
branch point. The classes of the corresponding  \Teichmuller curves
have been determined in \cite{bainbridge07}.

\section{Origamis, Square-tiled surfaces and their spin structure}
\label{sec:origamis}

Let $\omoduli[g]$ be the moduli space of flat surfaces $(X,\omega)$ and for
any partition $\kappa$ of $2g-2$, let $\omoduli[g](\kappa)$ be
the stratum, where the divisor of $\omega$ has type $\kappa$. In this
paper $(X,\omega)$ will always be an arithmetic Veech surface of genus $g > 1$. 
This is equivalent to requiring the existence of an \textit{origami map},
a covering $p:X\to E$ to an elliptic curve $E$ such that $p$ is branched over 
only one point and $\omega = p^*\omega_E$. 
The map $p$ is unique only up to isogeny and
translation on $E$ (the latter can be dispensed with by translating
the unique branch point to the origin). We call $p$ \textit{reduced}, if
it does not factor over an origami map $p':X\to E'$ 
that has strictly smaller degree. Equivalently, $p$ is reduced,
if and only if the lattice of generated by relative periods 
\[\Per(\omega) = \set{\int_\gamma \omega}{\gamma\in H_1(X,Z(\omega),\ZZ)}\subset
\CC\]
is equal to $\Per(\omega_E) = \{\int_\gamma\omega_E: \gamma\in H_1(E,\ZZ)\}$.
\par
If $E$ is the particular elliptic curve with $j(E) = 1728$, then
$X$ is called \textit{square-tiled surface}. In this case,
$\Per(\omega)\subset \ZZ\oplus i\ZZ$.
\par
A covering $q:X\to E'$
to an elliptic curve $E'$ is called \textit{minimal}
or \textit{optimal}, if 
it does not factor over an isogeny of degree $> 1$. 
A covering is minimal, iff the induced map $q_*$ on the
first absolute homology is surjective.
\par
Let $E'[2] = \{P_0,P_1,P_2,P_3\}$ 
denote the set of $2$-torsion points, where $P_0 = 0$,
and let $\sigma\in \Aut(X,\omega)$ denote the hyperelliptic involution.
Let $W_X$ denote the Weierstra\ss{} divisor on $X$. From now on we
restrict to the case of genus two surfaces.

\begin{prop} \label{prop:kani_norm}
 For any arithmetic Veech surface of genus $2$, there is a reduced
 origami map $p:X\to E$ and a decomposition $p = \iota \circ q$
 into a minimal covering $q:X\to E'$ of degree $d$ and an isogeny $\iota:E'\to
E$
 of degree $\tors\geq 1$.
 \par
 The map $q$, and a fortiori $p$, is uniquely determined by the requirement
 that 
 \[q_*W_X = \begin{cases}
             2(P_1 + P_2 + P_3), &\text{if}\ d \congruent 0 \bmod 2\\
             3P_0 + P_1 + P_2 + P_3, &\text{if}\ d \congruent 1 \bmod 2\,. 
            \end{cases}
 \]
\end{prop}
\par
\noindent
We call the origami map $p$ with a factorization and 
location of branch points as in this proposition \textit{normalized}.
\par
\begin{proof}
 By \cite[Proposition 2.2]{Ka03}, there is a uniquely determined minimal,
 normalized covering $q:X\to E'$. Moreover, this covering satisfies
 \[ [-1] \circ q = q\circ \sigma\]
 and since the ramification points of $q$ are not fixed by $\sigma$,
 their images $P$, $Q$ satisfy $[-1]Q = P$. Let $\iota:E'\to E$ be an
 isogeny with $\iota(P) = \iota(Q) = [-1]\iota(P)$, or
 equivalently $\iota([2]P) = 0$. Such an isogeny exists
 since $(X,\omega)$ is a Veech surface, and hence $P-Q$ is of
 finite order. The minimal such is given by the quotient map $E'\to E'/T$,
 where $T$ is the subgroup generated by $[2]P$.
\end{proof}
\par
It is possible that $\tors = 1$. In this case, the branching divisor
is non-reduced, \ie $P=Q \in E'[2]$. The integers $d$ and $\tors$ are uniquely
determined by the Veech surface. We call $d =d(X,\omega)$ the \textit{degree}
and
$\tors = \tors(X,\omega)$ the \textit{torsion order} of $(X,\omega)$.

\subsection{Spin structure} \label{subsec:spin}
Let $(X,\omega)\in \omoduli[2]$ be an arithmetic Veech surface with reduced,
normalized covering $p:X\to E$. A Weierstra\ss{} point $\tilde P$ is called
\textit{integral}, if
$p(\tilde P)$ is equal to the branch point of $p$. The number of integral
Weierstra\ss{} points is an invariant of the $\SL_2(\RR)$-orbit of
$(X,\omega)$, called the {\em spin invariant} $\spin(X,\omega)$. Depending on 
the parity of $d$ and $\tors$, we determine when it distinguishes orbits.
\par
Let $p:X\to E$ factorize as $p = \iota \circ q$ with a minimal, normalized
covering $q$ and an isogeny $\iota$ of degree $\tors \geq 1$. Let $P$ denote one
of the branch points of $q$. Then $\iota(P) \in E[2]$. If $\tors \congruent
1\bmod 2$, then the induced map $\iota[2]$ on the $2$-torsion points is an 
isomorphism. Thus, if $d\congruent 1\bmod 2$
\[\spin(X,\omega) = \begin{cases}
                     3, &\text{if}\ \iota(P) = 0\\
                     1, &\text{if}\ \iota(P) \neq 0
                    \end{cases}
\]
and if $d\congruent 0 \bmod 2$, then
\[\spin(X,\omega) = \begin{cases}
                     0, &\text{if}\ \iota(P) = 0\\
                     2, &\text{if}\ \iota(P) \neq 0
                    \end{cases}
\]
If on the other hand, $\tors \congruent 0 \bmod 2$, then $P$ is a
primitive $2M$-torsion point, and the fiber of $\iota(P)$ does not contain a
$2$-torsion point, since the equation $P + [2k]P = -P - [2k]P$ has no solution
$k\in \ZZ$. Thus in this case
\[\spin(X,\omega) = 0.\]
Note that the preceding discussion applies both to arithmetic
Veech surfaces in  $\omoduli[2](1,1)$ and to arithmetic
Veech surfaces in  $\omoduli[2](2)$. In the second case $M=1$ of course.
\par
Next we consider the case that $X$ is a {\em reducible genus two surface} but with compact
Jacobian, i.e.\ $X = E_1 \cup E_2$ is the union of two elliptic curves joined
at a node $S$. In this case an {\em origami map} $p:X \to E$ is simply defined to be 
a map that is non-constant on both factors, or equivalently $\omega = p^*\omega$
is non-zero on both components.  This implies that $E_1$ and $E_2$ (and $E$)
are isogenous. If $d_i = \deg(p|_{E_i})$ then obviously $d =\deg(p) = d_1 + d_2$.
We call Weierstra\ss\ divisor $W_X$ on $X$ the set of fixed points different 
from~$S$ of the elliptic involutions on $E_1$ and $E_2$ with respect to the
zero~$S$. Obviously $|W_X| = 6$
as in the smooth case. This notion is justified since one easily checks that
for any family of flat surfaces $(X_t, \omega_t)$ degenerating to $(X,\omega)$, 
the Weierstra\ss\ divisor $W_{X_t}$ converges to $W_X$. Again we let 
$\spin(X,\omega)$ be the number of integral Weierstra\ss\ points, i.e.\ the 
number of points in $W_X$ with image equal to $p(S)$. 
\par
There are no integral  Weierstra\ss\ points on a component $E_i$ iff $d_i$ is
odd.
If $d_i$ is even, there is three or one Weierstra\ss\ point, depending on
whether $p|_{E_i}$ factorizes through multiplication by two or not. The latter
can happen only if $d_i$ is divisible by four. For $d\congruent 1\bmod 2$
consequently
\[ \spin(X,\omega) \in \{1,3\}, \]
since  precisely one of the $d_i$ is odd. If $d$ is even, then both $d_i$ might
be
odd, resulting in no integral  Weierstra\ss\ points. If both $d_i$ are even
and one of the maps $p_i$ factors through multiplication by two, then $p$
factors through a two-isogeny.  Consequently, if $p$ is a reduced origami map
and $d\congruent 1\bmod 2$, then 
\[ \spin(X,\omega) \in \{0,2\}. \]
where $\spin(X,\omega)=0$ corresponds to both $d_i$ odd. 

\section{Pseudo-Hilbert modular surfaces}
\label{sec:pseudoquadratic orders}

In this section we introduce the surfaces containing the \Teichmuller curves
we are interested in. These are moduli spaces for Abelian surfaces with 
multiplication by pseudo-quadratic orders that we call {\em
pseudo-Hilbert modular surfaces} $\HMSoodual$. They admit a finite cover, 
which is a product of two modular curves. Consequently, many line bundles on
$\HMSoodual$ arise from line bundles on the modular curves and we summarize the
main properties.
Next, we introduce the \Teichmuller curves on  $\HMSoodual$
and fix notation for all the divisors on  $\HMSoodual$ we need.
See also \cite{bainbridge07}, \cite{Hermann} or \cite{mcmullenhilbert} for basic properties
of pseudo-Hilbert modular surfaces.
\par

\subsection{Modular curves and modular forms} \label{sec:modularcurves}

We let $\Gamma(d) \subset \SL_2(\ZZ)$ be the principal congruence group of level
$d \in \NN$ and $X(d)^\circ = \HH/\Gamma(d)$ be the (open) {\em modular curve}.
Its smooth compactification is denoted by $X(d)$. If $d\geq 3$, the curve
$X(d)$ has $\nu_{\infty,d} = \tfrac{[{\Gamma}(1):{\Gamma}(d)]}{2d}$
cusps $R_{d,j}$ and genus $g(X(d)) = 1 + \frac{d-6}{24d}\,|\SL_2(\ZZ/d\ZZ)|$.
\par
We record that $X(d) \to X(1)$ is a covering of degree 
\[ \Delta_d := |\SL_2(\ZZ/d\ZZ)| = [\Gamma(1):\Gamma(d)]
 = d^3 \prod_{p\divides d} (1-p^{-2})\]
if we consider these curves as quotient stacks. (In terms
of coarse moduli spaces, if we let $\ol{\Gamma}(d)$ denote the image
of $\Gamma(d)$ in $\ol{\Gamma}(1) = \PSL_2(\ZZ)$,
the covering is of degree $[\ol{\Gamma}(1):\ol{\Gamma}(d)]$, 
which is half the degree above for $d \geq 3$.)
\par
\smallskip
The {\em Hodge bundle} on $X(d)$ is $\lambda = \varpi_*(\omega_{E(d)/X(d)})$, 
where $\varpi: E(d) \to X(d)$ is the (compactified) universal family
(see Section~\ref{sec:toroidal}). We also write 
$\lambda_{X(d)}$ if we want to emphasized the level. Global sections of
$\lambda_{X(d)}^{\otimes k}$ are modular forms of weight $k$ for $\Gamma(d)$. 
Moreover, $\lambda_{X(d)}^{\otimes 2} = K_{X(d)}(R_d)$, where $R_d$
is the divisor of cusps and $K_{X(d)}$ is the canonical bundle.
\par
The discriminant $f_\Delta$ is a modular form of weight $12$ for
$\Gamma(1)$. It is non-zero on $X(d)^\circ$ and vanishes to the order $d$ at
each cusp $R_{d,j}$ ($j=1,\dots,\nu_{\infty,d}$) of $X(d)$. 
Thus 
\begin{equation} \label{eq:lambdaR}
12\lambda_{X(d)} = d\cdot R_d.
\end{equation}
\par
The principal congruence group of level $d$ is conjugate to another 
congruence group
$$ \Gamma(d)_d = \diag(d,1)\cdot \Gamma(d) \cdot \diag(d^{-1},1).$$
Consequently, the action of $\Gamma(d)_d$ and $\Gamma(d)$ on $\HH$
are equivariant with respect to the multiplication map by $d$ on $\HH$
and there is an isomorphism
$$ X(d)^\circ = \HH/\Gamma(d) \cong \HH /\Gamma(d)_d =: X(d)_d^\circ.$$ 
The two-fold product of the groups $\Gamma(d)_d$ appears naturally
as subgroup of pseudo-Hilbert modular groups, as we will see next.
\par

\subsection{Pseudo-Hilbert Modular surfaces} \label{sec:pseudoHMS}

Let $d\in \NN$ and $D= d^2$. Following the conventions for 
Hilbert modular surfaces, we let $K = \QQ\oplus \QQ$, whose subring
\[\frako_{d^2} = \{x = (x',x'') \in \ZZ\oplus \ZZ: x' \equiv x'' \bmod d \} \subset K\]
will be called a pseudo-quadratic order of discriminant $D$. 
Let $\frako_{d^2}^\dual = \tfrac1{(d,-d)} \frako_{d^2}$ be the inverse different. 
The pseudo-Hilbert modular group is
\bes \SLoodual \= \SL(\frako_{d^2}\oplus\frako_{d^2}^\dual)
\ees
and {\em pseudo-Hilbert modular surface} is the quotient\footnote{topologically, but not as a quotient stack, see Section~\ref{sec:quotientstack}}
\[\oHMSoodual = \HH^2 / \SLoodual.\]
It is the moduli space parameterizing abelian surfaces 
with multiplication by the pseudo-quadratic order of
discriminant $d^2$ as we will see in Section~\ref{subsec:modular_embeddings}. 
Since
\[ \Gamma(d)_d^2 \subset \SLoodual \subset \Gamma(1)_d^2,\]
both inclusions being of degree
$|\SL_2(\ZZ/d\ZZ)|$, the pseudo-Hilbert Modular surface
admits a useful covering given by
\[\tau: (X(d)_d^\circ)^2 \to \oHMSoodual\]
and a quotient given by
\[\beta: \oHMSoodual \to (X(1)_d^\circ)^2.\] 
The factor group $\Gamma(1)_d^2/\Gamma(d)_d^2$, and
thus a fortiori $\SLoodual/\Gamma(d)_d^2$, acts on the smooth
compactification $\Modulielldleveld^2$ of $(\oModulielldleveld)^2$
and the quotient maps $\tau$ and $\beta$ extend to quotient maps 
\[\tau :\Modulielldleveld^2 \to  \HMSoodual, \qquad \text{and}\qquad 
\beta :\HMSoodual \to X(1)_d^2.\]
We will work with this normal (but not smooth) compactification 
$\HMSoodual$ of $\oHMSoodual$. In fact $\HMSoodual$ is
the Baily-Borel compactification of $\oHMSoodual$. We now list the divisors
on  $\HMSoodual$ that will be important in the sequel.
\par
\medskip
\paragraph{\textbf{Boundary divisors}}
The image of $(\ol{\HH} \setminus \HH) \times \HH$  is
a curve ${R^{(1),\circ}} \subset  \HMSoodual$ and the image of 
$\HH\times (\ol{\HH} \setminus \HH)$ is a curve ${R^{(2),\circ}} \subset 
\HMSoodual$. Their closures are denoted by $R^{(i)}$.
The curves ${R^{(i),\circ}}$ are irreducible and isomorphic to 
$\HH/\Gamma_1(d)^\pm$ (\cite[Proposition~2.4]{bainbridge07}).\footnote{There 
are different indexing conventions for the boundary divisors in 
\cite{bainbridge07} and in \cite{Hermann}.
As mnemonic for our convention, keep in mind that $R^{(i)}$ and 
$\lambda_i$ are pulled back via $\pr_i$.}
\par
\medskip
\paragraph{\bf The Hodge bundles.} We let $\lambda^{(i)}_\Box = 
\pr_i^* \lambda_{X(d)}$ be the pullback of the Hodge bundle to the product. 
The next important
divisor classes on $\HMSoodual$ are the Hodge bundles
\[\lambda_i = (\pr_i \circ \beta)^* \lambda_{X(1)}.\]
By definition $\tau^* \lambda_i = \lambda^{(i)}_\Box$.
\par
In the same way, we define $R^{(i)}_\Box = \pr_i^*R_d$ as the pullback 
of the boundary divisors to $X(d)_d^2$. They consist of $\nu_{\infty,d}$
irreducible components $R^{(i)}_{\Box,j}$, $j=1,\dots,\nu_{\infty,d}$.
\par
Pulling back \eqref{eq:lambdaR} to the product  $X(d)_d^2$ and then taking
its $\tau$-pushforward we obtain the important relation
\be\ \label{eq:lambdaRonX}
R^{(i)} = \frac{12}{d}\lambda_i.
\ee
in $\Pic(\HMSoodual)$.
\par
\medskip
\paragraph{\bf The product locus.} We denote by $P_{d^2}^\circ$ the {\em
product locus}, the locus of abelian surfaces that split as a polarized
surface. We will determine the class of this locus in Section~\ref{sec:redlocus}.
The complement $\oHMSoodual \setminus P_{d^2}^\circ$ consists of principally
polarized abelian surfaces that are Jacobians of genus two curves.
\par
\medskip
\paragraph{\bf The Teichm\"uller curves.} The projection of an $\SL_2(\RR)$-orbit of 
a square-tiled surface $(X,\omega)$ is a \Teichmuller curve $C$ in $\moduli[2]$. 
If $q: X \to E$ is a minimal torus covering of degree $d$, then the kernel of 
$\Jac(q): \Jac(X) \to E$ is a connected abelian subvariety of exponent~$d$
(cf.~\eqref{eq:defexp}) by \cite[Lemma 12.3.1, Corollary~12.1.5 and Proposition~12.1.9]{bl}.
Consequently, by Proposition~\ref{prop:HMSismodulispace} below, a square-tiled surface 
that factorizes through such a map $q$ defines a point in $\HMSoodual$ and the 
corresponding \Teichmuller curve $C$ is a curve in  $\HMSoodual$.
\par
We let $W_D$ ($D=d^2$) be the union of \Teichmuller curves generated by reduced
square-tiled surfaces
of degree $d$ where $\omega$ has a double zero. By the results in the
preceding section, $W_D$ decomposes into spin components $W_D^\spin$. 
The topology of $W_D$ is completely determined 
by the work of \cite{mcmullenspin}, \cite{bainbridge07}, 
and \cite{mukamelorbifold}. In particular the spin components 
are irreducible.
\par
We let $T_{d,\tors}$  be the union of
\Teichmuller curves generated by reduced square-tiled surfaces
of degree $d$ such that $\omega$ has two simple zeros and $(X,\omega)$ 
has torsion order $\tors$. By the preceding section, $T_{d,\tors}$
decomposes into its spin components $T_{d,\tors,\spin}$.

\subsection{On quotient stacks} \label{sec:quotientstack}

Since we suppose $d \geq 3$ throughout, the stack discussion on $X(1)$
in the beginning of this section was inessential. The group
$\SLoodual$ however contains for all $d$ an element of finite order that
acts trivially on $\HH^2$, namely $-I$ embedded diagonally. We want
the main object of our studies, the pseudo-Hilbert modular surface
$\HMSoodual$ to be a variety, rather than a stack with global 
non-trival isotropy group of order two. For this purpose we consider
$\HMSoodual$ as the quotient stack $\HH^2 / \PP\SLoodual$. As a set, 
$\oHMSoodual = \HH^2 / \SLoodual$, as introduced above, but the
morphism $\tau$ is of degree $|\PSL_2(\ZZ/d\ZZ)| = \Delta_d/2$ 
throughout this paper. In particular, it is also possible
to define the Hodge bundles 'from above' without invoking the orbifold bundles
on $X(1)$ by the relation $\lambda_i = \tfrac2{\Delta_d}\tau_*
\lambda^{(i)}_\Box$.
The equation~\eqref{eq:lambdaRonX} holds with this convention (and with
the reduced scheme structure on $R^{(i)}$).
\par
The reason for this discussion is that the diagonally embedded $-I$
does no longer act trivially when considering the universal family, 
see~\eqref{eq:actHilbertsemidir} in the next section. 
So there is no choice but to let the
universal family $ \oFamilyoodual$
and its compactification be really the quotient stack by the group
$\Semidirectoodual$. In particular, the map 
$\tilde\tau$ is of degree $\Delta_d d^2$.
This has the irritating consequence that the  map of the
universal family  $\pi^\circ: \oFamilyoodual  \to \oHMSoodual$
is the composition of the forgetful map
$\HH^2\times \CC^2/\Semidirectoodual \to \HH^2/\SLoodual$ composed
with a (pointwise identity) map  $\HH^2/\SLoodual \to \HMSoodual$
of degree $\frac12$. This factor has to be taken into account in 
push-forwards, see Section~\ref{sec:TMing2}.

\section[]{Abelian surfaces with multiplication by pseudo-quadratic orders and modular embeddings}
\label{subsec:modular_embeddings}
Here, we sketch how $\oHMSoodual$ parametrizes abelian surfaces with
multiplication by $\frakod$ and describe the universal family
\be \label{eq:univ_oHSMS}
\pi^\circ: \oFamilyoodual \= \HH^2\times \CC^2/\Semidirectoodual  \to
\oHMSoodual
\ee
where 
\be
\Semidirectoodual \= \SL(\frako_{d^2}\oplus\frako_{d^2}^\dual)
			\semidirect (\frako_{d^2}^\dual \oplus \frako_{d^2})
\quad \subset \SL_2(K)\semidirect K^2.
\ee
One should be aware that $\oFamilyoodual\to \oHMSoodual$ is the universal family
only when considered as a quotient stack. The fibers of the underlying variety
are Kummer surfaces, and in particular singular. Nevertheless, the open family
and its compactification, introduced in Section~\ref{sec:toroidal}, are
both quotients of smooth varieties by finite groups and thus smooth when
considered as stacks.
\par
It will be convenient to compare this family to the universal family of all
principally polarized abelian surfaces via a map $\tilde\psi : \HH^2\times \CC^2
\to \HH_2\times \CC^2$ that is equivariant with respect to a group inclusion
$\tilde\Psi: \Semidirectoodual \to \Sp(4,\ZZ)\semidirect\ZZ^4$.
Such a pair $(\tilde\psi,\tilde\Psi)$ is sometimes called {\em modular
embedding} and it will be used in the next section to pull back theta functions.
\par
Recall that the {\em  exponent} $e(Y)$ of an abelian subvariety $Y$ 
of dimension $r$ in a principally polarized abelian variety $(A,\Theta)$ 
is defined as
\be \label{eq:defexp}
e(Y) = d_r, \quad \text{if $\Theta|_Y$ has type $(d_1,\ldots,d_r)$,} 
\ee
see \cite[Section~1.2 and 12.1]{bl}.
\par
\begin{prop}
\label{prop:HMSismodulispace}
The pseudo-Hilbert modular surface surface $\oHMSoodual$ is the moduli space of
all pairs $(A,\rho)$, where $A$ is a principally polarized abelian surface and
$\rho:\frakod \to \End(A)$ is a choice of multiplication by $\frakod$.
\par
Equivalently, $\oHMSoodual$ is the moduli space of all pairs consisting of
a principally polarized abelian surface~$A$ together with a projection 
$q: A\to E$ to an elliptic curve $E$ such that $\ker(q)$ is a connected 
abelian subvariety of exponent~$d$.
\end{prop}
\par
For the convenience of the reader and to fix notations, we provide a sketch of
the proof the first statement, following {\cite[Theorem 2.2]{bainbridge07}}.
The second statement follows from \cite[Proposition~12.1.1 and 
Proposition~12.1.9]{bl} after unwinding the definitions.
\par
We want to provide $\oodual$ with a polarization. For this purpose
we define the  'Galois conjugation' on $\frako_{d^2}$ by $(x',x'')^\sigma = (x'',x')$.
With the usual definition of trace the pairing
\[\skp{(x_1, y_1)}{(x_2,y_2)} = \trace(x_1y_2-x_2y_1).\]
on $\oodual$ is unimodular, alternating and $\ZZ$-valued, hence a polarization.
Moreover, we let $\sqrt{D} = (d,-d) \in K$.  Then, 
a symplectic basis of $\oodual$ is
\[a_1=(\tfrac1{\sqrt{D}}\eta_2^\sigma,0),\quad 
a_2=(-\tfrac1{\sqrt{D}}\eta_1^\sigma,0),\quad b_1=(0,\eta_1),\quad b_2=(0,\eta_2),\]
where $\eta_1,\eta_2$ is an arbitrary basis of $\frako_{d^2}$.
For $z = (z_1,z_2)\in \HH^2$, define the embedding
\[\oodual \to \CC^2, \qquad (x,y) \mapsto \tvector{x'z_1 + y'}{x''z_2+ y''}.\]
The image is a lattice in $\CC^2$ spanned by the columns of
\[ \Pi_z = \begin{pmatrix}
   \tfrac1{d}\eta_2''z_1 & -\tfrac1{d}\eta_1''z_1 & \eta_1'  & \eta_2'\\
   -\tfrac1{d}\eta_2'z_2   & \tfrac1{d}\eta_1'z_2 & \eta_1'' & \eta_2''
  \end{pmatrix} = ( z^*\cdot A^T,B)\]
where $z^* = \tmatrix{z_1}00{z_2} $ and where $B = \tmatrix{\eta_1'}{\eta_2'}{\eta_1''}{\eta_2''}$ and $A = B^{-1}$.
We will work throughout with the choice 
\begin{align*}
 B = \tmatrix{1}{0}{1}{d},\quad \text{hence} \quad 
A = \tmatrix{1}{0}{-\tfrac1d}{\tfrac1d}.
\end{align*}
\par
The quotient $A_{d^2,z} = \CC^2/\Pi_z\ZZ^4$ is a principally polarized abelian
surface (ppas), 
polarized by the hermitian form with matrix $\Im(z^*)^{-1}$ and the
columns of $\Pi_z$ are a symplectic basis for the pairing with matrix
$\tmatrix{0}{I_2}{-I_2}{0}$. The associated point in $\HH_2$ is $Z = A\cdot z^*
\cdot A^T$, with the convention that $Z\in \HH_2$ corresponds to the ppas with
lattice spanned by the columns of $(Z,I_2)$. It admits multiplication by
$\frako_{d^2}$ via the diagonal action on the embedding $\oodual \to \CC^2$.
This justifies the claims made in Section~\ref{sec:pseudoHMS}.
\par
Since both eigenspaces of multiplication by $K$ 
are defined over $\QQ$, the abelian surface is isogenous 
to a product of elliptic curves with an isogeny of degree $d^2$. We give an
explicit basis of the sublattice 
corresponding to the product decomposition. 
It is generated by the columns of
\[\Pi_z \cdot \begin{pmatrix}
             B^T & 0 \\ 0 & d\cdot A
            \end{pmatrix} = 
 \begin{pmatrix}
  z_1 & 0   & d & 0\\
  0   & z_2 & 0 & d   
 \end{pmatrix}
\]
For an $\RR$-basis $(w_1,w_2)$ of $\CC$, define 
the elliptic curve $E_{w_1,w_2} = \CC/(w_1\ZZ + w_2\ZZ)$.
Then the isogeny between abelian varieties
\[E_{z_1,d} \times E_{z_2,d}  \longrightarrow A_{d^2,z}\]
is induced by the identity on the universal cover. The coordinate projections 
$p_i:\CC^2\to \CC$, $i=1,2$ induce the dual isogeny 
\[A_{d^2,z} \longrightarrow E_{z_1/d,1} \times E_{z_2/d,1}\]
which after composition with the isomorphism covered by $\CC^2\to \CC^2$, 
$z\mapsto d\cdot z$ becomes multiplication by $d$ on $E_{z_1,d}\times
E_{z_2,d}$.
\par
This completes the sketch of the proof of
Proposition~\ref{prop:HMSismodulispace}. 
\par
\medskip
\paragraph{\textbf{Modular embeddings}}
The universal family is now easily
obtained by pullback of the universal family of principally polarized abelian
surfaces over $\HH_2$ via a modular embedding.
\par
\begin{lemma}\label{lem:Siegel-embedding-equivariant}
 The embedding 
\[\tilde\psi : \HH^2\times \CC^2 \to \HH_2\times \CC^2, \quad (z,u) \mapsto (Az^*A^T,Au)\]
 is equivariant with respect to
\bas
\tilde\Psi:\,\,\,\, &\Semidirectoodual &\to &\quad \Sp(4,\ZZ)\semidirect\ZZ^4, \\
&(M,r) &\mapsto &\quad S\cdot (M^*,r)\cdot S^{-1} = (\begin{pmatrix}
                                                 Aa^*B & Ab^*A^T\\
                                                 B^Tc^*B & B^Te^*A^T
                                                \end{pmatrix}, (r_1B,r_2A^T))
\eas
where $M^* = \tmatrix{a^*}{b^*}{c^*}{e^*}$, $r=(r_1,r_2)$ and 
$S = (\diag(A,B^T),0)\in \Sp(4,\QQ)\semidirect\QQ^4$.
\par
\end{lemma}
\par
Note that the induced map $\oHMSoodual \to \cA_2$ does not depend on the choice
of the matrix $B$.
If $B'$ is another basis, $A' = B'^{-1}$, and $(\tilde\psi',\tilde\Psi')$ is the
embedding associated with~$B'$, then
\[\tilde\psi' = g\circ \tilde\psi\quad \text{and}\quad \tilde\Psi' = g\cdot
\tilde\Psi \cdot g^{-1}
 \qquad \text{where}\ g = \diag(A'B, B'^TA^T)\in \Sp(4,\ZZ)\]
\par
The proof of Lemma~\ref{lem:Siegel-embedding-equivariant} is a straightforward calculation, once one fixes the precise definition of the group actions on source and target. 
We define the semidirect products $\Sp(2g,\RR)\semidirect\RR^{2g}$
by the rule
\[(M_1,r_1)\cdot (M_2,r_2) := (M_1M_2, r_1M_2 + r_2).\]
This semidirect product acts on the product $\HH_g\times \CC^g$ by
\be \label{eq:semidir_act}
(Z,v) \mapsto (M(Z), ((CZ + E)^T)^{-1}(v+ (Z,I_g)r^T))
\ee
where $M = \tmatrix{A}{B}{C}{E}$ and $r\in \ZZ^{2g}$, and $M(Z) = (AZ+B)(CZ+E)^{-1}$. 
The action is compatible with the projection on the first factor and
standard action of  $\Sp(2g,\RR)$ on $\HH_g$.
\par
Next, we explicitly write out the action of $\Semidirectoodual$ on $\HH^2\times\CC^2$, 
or more generally of $\SL_2(\RR)\semidirect \RR^4$ on $\HH^2\times\CC^2$, 
which is implicitly already given by~\eqref{eq:semidir_act} and the modular embedding.
For $\alpha = (\alpha_1,\alpha_2)\in \CC^2$, set $\alpha^* = \tmatrix{\alpha_1}{0}{0}{\alpha_2}$. 
Then $(M,r)\in \SL_2(\RR)\semidirect \RR^4$ acts via 
\be \label{eq:actHilbertsemidir}
(z,u) \mapsto (M(z), (c^*z^* + e^*)^{-1}(u + (z^*, I_2)r^T))
\ee
where $M = \tmatrix{a}{b}{c}{e}$, $r = (r_1,r_2)$, and $r^T = (r_1',r_1'',r_2',r_2'')^T$ and where
\[
M(z) = (az+b)(cz+e)^{-1} = 
\biggl(\frac{a'z_1+b'}{c'z_1+e'},\frac{a''z_2+b''}{c''z_2+e''}\biggr).
\]
\par

\section{Compactifying the universal family
over \texorpdfstring{$\HMSoodual$}{the pseudo Hilbert modular surface}}
\label{sec:toroidal}

We will compute the classes of the curves $T_{d,\tors,\spin}$ as the image of a locus
cut out in the universal family of abelian surfaces over the pseudo-Hilbert modular surface.
Over the open pseudo-Hilbert modular surfaces, this family is  described
as the quotient  (see Section~\ref{subsec:modular_embeddings})
\[\pi^\circ: \oFamilyoodual \= \HH^2\times \CC^2/\Semidirectoodual  \to
\oHMSoodual\]
To perform intersection calculations, we need to work on a compact space and 
the aim of this section is to describe explicitly such a compactification
of $\oFamilyoodual$. Our strategy is as follows. The universal family over
the modular curve has a simple compactification, by adding a '$m$-gon' of
rational curves at every
cusp, the simplest instance of a toroidal compactification. In order to reduce from 
$\oFamilyoodual$ to such a situation, we have to pass from  $\oHMSoodual$ to a
finite cover where this surface is a product, as explained in the previous section, 
and then to pass fiberwise to an isogenous abelian variety. 
\par
The aim of this section is to exhibit a compactification of $\oFamilyoodual$ by
describing the action of the $2$-step covering group on the product of two 
compactified universal elliptic curves. We thus present a compactification of 
$\oFamilyoodual$ as a quotient of a smooth compact variety
by a finite group action. Along with this, we introduce local
coordinates at the boundary that will be used to define bundle extensions in the
next section.
\par
For this purpose we note that $\Semidirectoodual$
has a normal subgroup that is equal to a product $\tilde\Gamma(d)_d^2$,
where 
\[\tilde\Gamma(d)_d 
 = \diag(d,1)\cdot(\Gamma(d)\semidirect d\ZZ^2)\cdot\diag(d^{-1},1).\]
The quotient $\HH^2\times \CC^2/\tilde\Gamma(d)_d^2$ is a product family
\[\varpi^\circ\times\varpi^\circ: (E(d)_d^\circ)^2 \to 
(\oModulielldleveld)^2,\]
in fact of two copies on a universal family of elliptic curves.
\par
As a general guide to the notation in the sequel, groups $\Gamma$ act on $\HH$
or $\HH^2$, while groups with a tilde are semidirect
products acting on $\HH \times \CC$ or $(\HH \times \CC)^ 2$.
\par  
\begin{theorem}
\label{thm:toroidalcomp}  There exists a proper, smooth 4-dimensional stack 
$\Familyoodual$ containing $\oFamilyoodual$ as a Zariski open subset such that
 \begin{enumerate}[a)]
  \item The canonical projection $\pi^\circ$ extends to a
  flat, proper morphism
  \[\pi: \Familyoodual \to \HMSoodual.\]
  \item The map $\tilde\tau^\circ: (\oFamilyelldleveld)^2 \to \oFamilyoodual$
induced
by
  the inclusion $(\tilde\Gamma(d)_d)^2 \subset \Semidirectoodual$ extends
  to a finite morphism of degree $\Delta_d d^2$
  \[\tilde\tau: (\Familyelldleveld)^2 \to \Familyoodual\]
  over $\tau: (\Modulielldleveld)^2 \to \HMSoodual$
  \item The scheme underlying the stack $\Familyoodual$ has at most 
quotient singularities.
 \end{enumerate}
\end{theorem}
\par
The following diagram gives an overview of the spaces and maps involved.
\begin{align}
\label{eq:compactif-diagram}
\begin{split}
 \xymatrix{      
(\Familyelldleveld)^2 \ar[d]_{\varpi\times\varpi} \ar[r]^{\tilde\tau} 
    & \Familyoodual \ar[d]^{\pi}\\
(\Modulielldleveld)^2 \ar[r]_{\tau}  & \HMSoodual
}
\end{split}
\end{align}
\par
In order to prove this theorem, we employ the 
usual toroidal compactification of a family of elliptic curves.
For $\ell \in \NN$ we define more the twisted level subgroup $\Gamma(\ell)_d =
\diag(d,1)\cdot \Gamma(\ell) \cdot \diag(d^{-1},1)$. We let
\begin{align*}
 \tilde\Gamma(\ell)_d &= 
\diag(d,1)\cdot(\Gamma(\ell)\semidirect \ell\ZZ^2)\cdot\diag(d^{-1},1).
\end{align*}
The quotient $\oModulielldleveld = \HH / \Gamma(d)_d$
is the moduli space of $d$-polarized elliptic curves with a 
level $d$-structure and
\be \label{eq:univEllLevd}
\varpi^\circ: \oFamilyelldleveld = \HH\times\CC / \tilde\Gamma(d)_d \to \oModulielldleveld
\ee
is the universal family over it if $d\geq 3$. (Here and everywhere in the 
sequel we do not discuss the supplementary stack issues arising when $d=2$.) 
In particular, $\oFamilyelldleveld$ and $\oModulielldleveld$ is smooth.
\par
The following statement is the point of departure for the compactification. It
is well-known (see e.g. \cite[Section I.2]{HulekConstantinWeintraub93}), 
but we give its proof below since we need the coordinates introduced
there later on.
\par
\begin{prop}
\label{prop:toroidal-ell} There exists a compactification of  $\oFamilyelldleveld$
to  a smooth, projective surface $\Familyelldleveld$ with the following properties. 
\begin{enumerate}[a)]
\item The projection $\varpi^\circ$ has an extension to a 
 flat, proper morphism $$\varpi: \Familyelldleveld \to \Modulielldleveld.$$
\item The boundary $\partial\Familyelldleveld$ consists of $d\cdot \nu_{\infty,d}$ 
 rational curves $D_{C,k}$, where $C$ is a cusp of $\Gamma(d)_d$ and $k\in
\ZZ/d\ZZ$. We have
 \[D_{C_i,k}.D_{C_j,l} = \begin{cases}
                      -2, & i=j, k=l\\
                       1, & i=j, k=l\pm 1\\
                       0, & \text{else}
                     \end{cases}
\]
\item There is an action of $\tilde\Gamma(1)_d/\tilde\Gamma(d)_d \isom \SL_2(\ZZ/(d)) \semidirect (\ZZ/(d))^2$ on $\Familyelldleveld$ extending the action on $\oFamilyelldleveld$.
\end{enumerate}
\end{prop}
\begin{proof}[Proof of Theorem~\ref{thm:toroidalcomp}]
Thanks to the last item, we can define quotients of $\Familyelldleveld^2$ by all
subgroups of
 $\bigl(\tilde\Gamma(1)_d/\tilde\Gamma(d)_d\bigr)^2$. Therefore, setting
\[\Familyoodual = (\Familyelldleveld)^2\ /\ (\Semidirectoodual /
\tilde\Gamma(d)_d^2)\]
immediately yields the claims of Theorem~\ref{thm:toroidalcomp}.
\end{proof}
\par
We also obtain a description of the boundaries of the compactification
$\Familyoodual$.
As for $\HMSoodual$ there are boundaries components where the first resp.\ the
second elliptic curve degenerates. While for each of them there is a $d$-gon
over every cusp in the $\Familyelldleveld \times \Familyelldleveld$, there
are only two boundary components $D^{(i)}$ for $i \in \{1,2\}$ 
on $ \Familyoodual$.
\par
More precisely, let $S$ be the set of equivalence classes of cusps of $\Gamma(d)_d$.
For $C\in S$, $k=0,\dots,d-1$, we define 
the following divisors 
\begin{align*}
D_{C,k}^{(1)} &= D_{C,k} \times \Familyelldleveld &
D_{C,k}^{(2)} &= \Familyelldleveld \times D_{C,k} 
\end{align*}
in $\Familyelldleveld\times \Familyelldleveld$.
Then the boundary components are as follows.
\par
\begin{cor} \label{cor:boundarycomps}
The boundary of $\Familyoodual$ consists of the two irreducible components of
codimension one
\begin{align*}
  D^{(i)} &= \tilde\tau(D^{(i)}_{C,k}) &i=1,2 
\end{align*}
where $C\in S$, $k\in \ZZ/d\ZZ$ are arbitrary.
\end{cor}
\par
\medskip

\subsection[]{Toroidal compactification of families of elliptic curves}
\label{subsec:fan}
Here, we describe the compactification of the universal family $\Familyelldleveld$ of elliptic curves, and thereby prove Proposition~\ref{prop:toroidal-ell}.
\par
Let $T = (\CC^*)^2$ with coordinates $\zeta$ and $q$. For each integer
$k$ we define an inclusion $T \to T_{\sigma_k} \cong \CC^2$, given by
\be \label{eq:zq_to_zqk}
(\zeta,q) \mapsto (\zeta_{k},q_{k}) = (\zeta q^{-k}, \zeta^{-1}q^{k+1}).
\ee
Inside each $T_{\sigma_k}$ we define the open set $T_{\xi_{k+1}} = \{q_k\neq
0\} = D(q_k)$ and we consider this as an open subset of $T_{\sigma_{k+1}}$ via 
\[T_{\xi_{k+1}} \to T_{\sigma_{k+1}},\qquad 
  (\zeta_{k},q_{k}) \mapsto (\zeta_{k+1},q_{k+1}) = (q_{k}^{-1},\zeta_{k}q_{k}^2)\, .\]
Gluing $T_{\sigma_{k}}$ to $T_{\sigma_{k+1}}$ along the open set 
$T_{\xi_{k+1}}$ gives an infinite chain of rational lines $D_{k+1}$. 
\par
The line $D_k$ is covered by two affine charts. It is given by
\[V(\zeta_{k-1})\subset T_{\sigma_{k-1}}\quad \text{and}\quad V(q_k) \subset T_{\sigma_{k}}\]
which are glued along $D(q_{k-1}) \leftrightarrow D(\zeta_k)$ by $q_{k-1} = \zeta_k^{-1}$.
As $\zeta_{k-1} = q_{k-1}^{-2}q_k$, this is indeed well-defined, and
moreover $D_k$ has self-intersection $-2$. (In fact, we described a
partial toroidal compactification of $T$, using
the collection $\sigma = \{\sigma_k\}_{k\in \ZZ}$ 
of rational polyhedral cones in $\RR^2$ defined by
\[\sigma_k = \RR_{\geq 0}\cdot (k,1) + \RR_{\geq 0}\cdot (k+1,1),\quad k\in\ZZ,\]
but we will not need this viewpoint. See \cite{HulekConstantinWeintraub93} 
for details.)
\par
\medskip
\medskip

We now compactify $\oFamilyelldleveld$ 
by adding suitable $d$-gons over the cusps of $\Gamma(d)_d$. 
We can carry this out for one cusp at a time, and in fact, it
suffices to describe a compactification for the cusps $\infty$, 
since $\Gamma(d)_d$ is normal in $\Gamma(1)_d$, which has only 
one cusp.
\par
\medskip
\paragraph{\textbf{Compactification over \texorpdfstring{$\infty$}{infinity}}}
We carry out the standard construction of a toroidal compactification. 
It will be convenient to
represent elements of the semidirect product $\Sp_{2g}(\RR)\semidirect
\RR^{2g}$ in matrix form via
\[ (M,r) \mapsto 
 \begin{pmatrix}
  1 & r\\
  0 & M\\
 \end{pmatrix}\, .   
\]
The stabilizer
$P = P_\infty(d)_d$  of a small neighborhood in  $\oFamilyelldleveld$ of the 
preimage of the cusp $\infty$ will have a normal subgroup $P^n = P^n_\infty(d)_d$ such that the quotient map by $P^n$ is given by a suitable coordinatewise
exponential map and such that the image
is isomorphic to $T$. On the partial compactification of $T$ defined above
the factor group $P^q = P/P^n $ acts, e.g.\ on the boundary curves by a shift of
indices. For each of the cusps these quotients are glued to the family over the
open
curve to obtain a compact space. In the sequel we need the precise form of the
coordinates, in particular~\eqref{eq:zeta_q} and~\eqref{eq:actzet_q_coord}.
\par
More precisely, let $N = \{\Im z \gg 1\}$ be a neighborhood of $\infty\in\HH$ not fixed by any
element outside the stabilizer of $\infty$ in $\Gamma(d)_d$. The preimage
$P_\infty(d)_d$ of the stabilizer of $N$ in $\tilde\Gamma(d)_d$ is equal to
\[P = P_\infty(d)_d = \{
 \begin{pmatrix}
  1 & \ZZ & d\ZZ\\
  0 &  1  & d^2\ZZ\\
  0 &  0  & 1
 \end{pmatrix}\}.\]
It contains the normal subgroup
\[P^n = P^n_\infty(d)_d =
\{\begin{pmatrix}
   1 & 0 & d\ZZ\\
   0 & 1 & d^2\ZZ\\
   0 & 0 & 1
  \end{pmatrix}\}.
\]
that acts on the $\varpi$-preimage of $N$, which is isomorphic to 
$N\times \CC$. The quotient map $N \times\CC \to N\times\CC/P^n$ is given by 
\be \label{eq:zeta_q}
(z,u)\mapsto (\zeta_\infty, q_\infty),\quad \text{with}\
\zeta_\infty = \e(\tfrac1{d}u),\quad q_\infty = \e(\tfrac1{d^2}z),
\ee
(where $\e(\cdot) = \exp(2\pi i\cdot)$) and identifies $N\times\CC/P^n$ with an open set $X_\infty$ in $T$.
We compactify $T$ as above and take $X_{\infty,\Sigma}$ to be the interior of the closure
of $X_\infty$. The boundary 
\[\partial X_{\infty,\Sigma} = X_{\infty,\Sigma}\setminus X_\infty\]
is an infinite chain of rational curves $D_{\infty,k}$.
\par
The group $P$ acts on $X_{\infty}$ through the factor group $P^q$ and the 
compactification is compatible with this action. In fact, the bigger group $P_\infty(1)_d$, the
preimage of the stabilizer of $N$ in $\tilde\Gamma(1)_d$, acts on
$T$, and thus on $X_{\infty,\Sigma}$, as the following lemma shows. Its proof is
a straight-forward calculation. Let  $\eta_d = e(1/d)$.
\par
\begin{lemma}\label{lem:boundary-action-2dims}
For $b\in \ZZ$, $s_i\in \ZZ$ and $\varepsilon\in \{\pm 1\}$, 
let 
\[\tilde g = \tilde g(s_1,s_2,\varepsilon,b) = 
\begin{pmatrix}
 1 & \tfrac1{d}s_1 & s_2 \\
 0 & \varepsilon   & bd\\
 0 & 0   & \varepsilon              
\end{pmatrix} \in P_\infty(1)_d =
\bigl\{
\begin{pmatrix}
 1 & \tfrac1{d}\ZZ & \ZZ \\
 0 & \pm 1 & d\ZZ \\
 0 & 0  & \pm 1                
\end{pmatrix}\bigr\},\quad .\]
The element 
\begin{enumerate}[a)]
 \item $\tilde g$ acts on the coordinates $(\zeta,q)= (\zeta_\infty,q_\infty)$
by
\ba \label{eq:actzet_q_coord}
 \zeta &\mapsto \zeta^\varepsilon \cdot q^{\varepsilon s_1} \cdot
\eta_{d}^{\varepsilon s_2}\\
 q     &\mapsto q \cdot \eta_{d}^{\varepsilon b}
\ea
\item $\tilde g$ acts on the coordinates $(\zeta_k, q_k) =
(\zeta_{\infty,k},q_{\infty,k})$ by
\begin{align*}
 \zeta_k &\mapsto \begin{cases}
                   \zeta_{k-s_1} \cdot \eta_{d}^{s_2 - bk}, & \varepsilon = 1\\
		   q_{-s_1-k-1} \cdot \eta_{d}^{bk-s_2}, & \varepsilon = -1
                  \end{cases},\
 q_k     &\mapsto \begin{cases}
                   q_{k-s_1} \cdot \eta_{d}^{(k+1)b - s_2}, & \varepsilon = 1\\
		   \zeta_{-s_1-k-1} \cdot \eta_{d}^{s_2 - (k+1)b}, & \varepsilon
= -1
                  \end{cases}
\end{align*}
\item $\tilde g$ acts on set of rational curves $D_{\infty,k}$ ($k\in\ZZ$) by 
\[D_{\infty,k} \mapsto 	D_{\infty,\varepsilon(k+s_1)}.\]
In particular, the action of $P_\infty(1)_d$ on
$\{D_{\infty,k}\}$ is transitive.
\end{enumerate}
\end{lemma}
The action of $P^q$ on $X_{\infty,\Sigma}$ is properly
discontinuous and free. Let $Y_{\infty,\Sigma} = X_{\infty,\Sigma}/P^q$
be the quotient. The action of $P^q$ identifies $D_{\infty,k}$ with
$D_{\infty,k+dr}$, $r\in \ZZ$, whence the boundary of the quotient
$Y_{\infty,\Sigma}$ consists of a $d$-gon of rational curves, also called
$D_{\infty,k}$ ($k\in \ZZ/d\ZZ$).
  \par
\medskip
\paragraph{\textbf{Compactification over an arbitrary cusp}}
Let $S$ be a system of representatives of the cusps of $\Gamma(d)_d$.
For $C\in S$, choose an element 
\[M_C = \tmatrix{\alpha_C}{\beta_C}{\gamma_C}{\delta_C}\in \Gamma(1)_d
\qquad \text{such that}\ \Gamma(d)_dM_C(\infty) = C.\] 
The neighborhood $N_C = M_C(N)$ of $C$ in $\HH$ is not fixed by an element
outside the stabilizer of $C$. We define $P_C(d)_d$ as the preimage of the
stabilizer of $N_C$, and $P_C(d)_d$ its unipotent radical. The coordinates
on the quotient $N_C\times \CC/P_C'(d)_d$ are
\be \label{eq:zeta_q_otherC}
\zeta_{C} = \e((-\gamma_Cz + \alpha_C)^{-1}\tfrac{u}{d}),\qquad 
q_{C} = \e(\tfrac{M_C^{-1}z}{d^2})\,.
\ee
As before, the image of $N_C\times\CC$ is an open set $X_{C}$ in the torus $T
=\Spec\CC[\zeta_C^{\pm},q_C^{\pm}]$ and, using the same torus embedding as
above, we compactify it by taking $X_{C,\Sigma}$ to be the interior of the
closure of $X_C$ in $T_\Sigma$. Again let $Y_{C,\Sigma} =
X_{C,\Sigma}/P_C(d)_d$ be the quotient. Let
$Y_C$ be the image of $X_C$ in $Y_{C,\Sigma}$. Then the map 
$i_C: Y_C \to \oFamilyelldleveld$
that sends an orbit of $P_C(d)_d$ to its $\tilde\Gamma(d)_d$-orbit is an
embedding.
\par
The space $\Familyelldleveld$ is now obtained by taking the disjoint union
\[\oFamilyelldleveld\ \dot\cup\ \dot{\bigcup}_{C\in S} Y_{C,\Sigma}\]
and dividing out the equivalence relation generated by
identifying $x\in\oFamilyelldleveld$ with $y \in Y_{C}$ if $i_C(y) = x$.
This completes the proof of Proposition~\ref{prop:toroidal-ell}.
\par

\subsection{Description of the boundaries of $\Familyoodual$}
In this section, we analyze the action of the quotient group $H_{d^2} = \Semidirectoodual / (\Gamma(d)_d)^2$ on the set of boundary components of $\Familyelldleveld^2$, showing the claims of Corollary~\ref{cor:boundarycomps}. 
Secondly, we determine local coordinates of a neighborhood of $D^{(i)}$ by showing that the isotropy group of a generic point is trivial.
\par
Recall the group isomorphisms
\[\olred^{(i)}: H_{d^2} \to \SL_2(\ZZ/d\ZZ)\semidirect (\ZZ/d\ZZ)^2,\quad
i=1,2\]
induced by
\[\red^{(i)}: \Semidirectoodual \to \SL_2(\ZZ/d\ZZ)\semidirect (\ZZ/d\ZZ)^2, 
 \quad (A,s) \mapsto \ol{\diag(d^{-1},1) \cdot (A^{(i)},s^{(i)}) \cdot \diag(d,1)}.\]
where $\ol{\cdot}$ denotes the reduction modulo $d$.
\par

\begin{lemma} \label{lem:Stab_of_Bdry_in_Hd2}
 The group $H_{d^2}$ acts transitively on 
 $\set{D_{C,k}^{(i)}}{C \in S, k\in \ZZ/d\ZZ}$ for each $i=1,2$. 
 The stabilizer of $D_{\infty,0}^{(i)}$ is given by
\[\olred^{(i)}(\Stab_{H_{d^2}}(D^{(i)}_{\infty,0})) = 
  \{[\tmatrix{\pm 1}{\ast}{0}{\pm 1},(0,\ast)]\}
  \subset \SL_2(\ZZ/d\ZZ)\semidirect (\ZZ/d\ZZ)^2\]
 and is of order $2d^2$. Moreover the pointwise stabilizer
\[\Stab_{H_{d^2}}(D^{(i)}_{\infty,0})\]
 is trivial.
\end{lemma}
\begin{proof}
By symmetry, we may focus on $i=1$. The group $\SLoodual$ acts transitively on the set $\set{C\times \Modulielldleveld}{C\in S}$, so it suffices to show that $\Semidirectoodual \cap (P_\infty(1)_d \times \tilde\Gamma(1)_d)$ acts transitively on $\set{D_{\infty,k}^{(1)}}{k\in \ZZ/d}$. 
For $s_1\in \RR^2$, let $h(s_1) = [I,(s_1,0)] \in \tilde G$. We have
\[h(\tfrac1{\sqrt{D}}) = [I,((\tfrac1{d},-\tfrac1{d}),0)] \in \Semidirectoodual \cap (P_\infty(1)_d \times \tilde\Gamma(1)_d),\]
which maps $D_{\infty,k}^{(1)}$ to $D_{\infty,k+1}^{(1)}$.

Concerning the stabilizer group of $D^{(1)}_{\infty,0}$, we have
\[\olred^{(i)}\bigl(\Stab_{H_{d^2}}(D^{(1)}_{\infty,k})\bigr) = 
\red^{(i)}\biggl(\Semidirectoodual \cap
\bigl(\Stab_{P_\infty(1)_d}(D_{\infty,k}) \times
\tilde\Gamma(1)_d\bigr)\biggr).\]
Using this observation and Lemma~\ref{lem:boundary-action-2dims}, one can easily
determine the stabilizer and the pointwise stabilizer.
\end{proof}
\par

\paragraph{\textbf{Local coordinates at the boundaries}}
We describe local coordinates in the neighborhood of a point $x\in D^{(i)}$
($i=1,2$). These will be used to extend the line bundles in the next section.
\par
For $i=1,2$ and $k\in \ZZ$, we introduce, following
\eqref{eq:zq_to_zqk} and \eqref{eq:zeta_q}, the notations
\begin{align} \label{eqn:localcoords_at_boundaries_of_A}
 \zeta_{i}   &= \e(\tfrac1d u_i) &     q_i &= \e(\tfrac{1}{d^2}z_i)\\
 \zeta_{i,k} &= \zeta_i q_i^{k} &    q_{i,k} &= \zeta_i^{-1} q_i^{k+1}
\end{align}
It will be helpful to keep in mind the relations
\begin{equation} \label{eqn:zeta_ik_q_ik_to_zeta_q}
 \zeta_{i} = \zeta_{i,k}^{k+1} q_{i,k}^k,\qquad q_{i} = \zeta_{i,k}q_{i,k}
\end{equation}
Note also that we work throughout over the cusps $\infty$, but we suppress this
from the notation.
\par
\begin{lemma} \label{lem:localcoords_at_boundaries_of_A}
 Let $x\in D^{(i)}$ be a generic point 
and let $\tilde x$ be a lift of $x$ in $D_{\infty,k} \times
\HH\times \CC$ in case $i=1$, respectively in $\HH\times\CC \times D_{\infty,k}$
in case $i=2$. Then
\begin{align*}
 &(\zeta_{1,k},\ q_{1,k},\ z_2,\ u_2)  & i&=1\\
 &(z_1,\ u_1,\ \zeta_{2,k},\ q_{2,k})  & i&=2
\end{align*}
are local coordinates at $x$, in the sense that there exists an open
neighborhood $\tilde U$ of $\tilde x$ such that the canonical projection $U\to
\Familyoodual$ is a homeomorphism.
\end{lemma}
\par
In particular, the generic point of $D^{(i)}$ is smooth.
\begin{proof}
By symmetry, we may restrict to the case $i=1$. Since the action is properly
discontinuous, it suffices to show that a generic $\tilde x$ is not fixed by any
element $g \in \Semidirectoodual \setminus (P'_\infty(d)_d \times \{1\})$.
Let us write $g = (M,r)$, $M= \tmatrix{a}{b}{c}{e}$, $r=(r_1,r_2)$
and suppose that it fixes $\tilde x$.
As $x$ is generic, $z_2$ is not a fixed point of $M''$
and thus $M'' = \pm I$. For the same reason, $u_2$ is not a half-integral
lattice point $\tfrac1{2d}\tilde z_2\ZZ + \tfrac1{2}\ZZ$, and thus $u_2 \mapsto
a''(u_2 + z_2r_1'' + r_2''))$ does not fix a neighborhood of $u_2$ unless $a'' =
1$, $r_1'' = r_2'' = 0$. Since $M'$ fixes a point in $X_{\infty,\Sigma}$,
it is of the form
\[M' = 
 \begin{pmatrix}
  1 & r_1' & r_2' \\
  0 & \varepsilon & b'\\
  0 & 0 &\varepsilon
 \end{pmatrix}
\]
The congruence condition together with $a'' = 1$ forces $\varepsilon = 1$.
Since $b'' = 0$ and $r_2'' = 0$, we have $b' \in d^2\ZZ$ and $r_2'\in d\ZZ$.
Moreover, $M'$ has to fix the component $D_k \subset X_{\infty,\Sigma}$,
which according to Lemma~\ref{lem:boundary-action-2dims} entails
$r_1' = 0$. Altogether, this shows $M \in P'_{\infty}(d)_d \times \{1\}$.
\par
Alternatively, one can argue that $(q_{1,k},\zeta_{1,k},z_2,u_2)$ provide local
coordinates about $D^{(1)}_{\infty,0}$ on $\Familyelldleveld$, and that the
pointwise stabilizer $\Stab_{H_{d^2}}(D^{(1)}_{\infty,0})$ is trivial.
\end{proof}

\section{Divisors and line bundles on
\texorpdfstring{$\Familyoodual$}{the universal family}}\label{sec:divoodual}

On the universal family over an (open) pseudo-Hilbert modular surface 
there is a natural collection of line bundles, the common generalization
of the pullback of Hilbert modular forms and classical elliptic Jacobi forms. These are
the called Hilbert Jacobi-forms.  Theta functions
will be the main instances of sections of these line bundles. Our
aim is to express the classes of these line bundles in the rational Picard group
$\Pic_\QQ(\Familyoodual)$ in terms of line bundles that are good for intersection
theory calculations: the Hodge bundles, the boundary divisors and
the pullbacks $N^{(i)}$ of the zero sections. 
\par
The main result of this section is the following. The 
notation will be explained in the rest of this section.
\par 
\begin{theorem}\label{thm:class_of_HJF}
 Let $f$ be a Hilbert-Jacobi form of weight $\kappa \in (\tfrac12\ZZ)^2$, 
 index $m\in \tfrac12\frakod$ and a multiplier of order $\ell$ for the group
 $\Semidirectoodual$. Then the class of $\div(f)$ in $\Pic_\QQ(\Familyoodual)$
is
 \begin{align}
  (\kappa_1 + \tfrac{2m'}{d}) \pi^*{\lambda_1} + (\kappa_2 +
\tfrac{2m''}{d})\pi^*\lambda_2 + 
	\tfrac{2m'}{d} N^{(1)} + \tfrac{2m''}{d} N^{(2)}\, .
 \end{align}
\end{theorem}
\par
Note that it is almost meaningless to speak of the class of a line
bundle defined by giving explicit automorphy factors on the open
family. If  $\Jacforms_{\kappa,m}$ is one extension to
the compactification, any twist $\Jacforms_{\kappa,m}(nD^{(i)})$
for any integral $n$ and a boundary component $D^{(i)}$ will also 
be an extension. 
The theorem becomes meaningful only together with the description 
of the behavior at the boundary (in terms of Laurent series in local
coordinates) given in \eqref{eqn:HJF_bundel}. 
For practical purposes, any other boundary
conditions would work as well: we have to correct by the
vanishing order at the boundary and the difference is independent of
any choices, see Theorem~\ref{thm:howtointersect} for our
application.

\subsection{Divisors in {the Picard group of the universal family}: The boundary and torsion sections.}
 
In this section we list some important divisor classes in
the compactified universal family   $\Pic_\QQ(\Familyoodual)$
over the pseudo-Hilbert modular surface. The classes of a Hilbert
modular forms can be expressed in these bundles. For later
use we also define the divisors corresponding to zero sections
and compare it to the divisor of torsion sections. 
\par
Recall from Section~\ref{sec:pseudoHMS} the definition of the Hodge bundles
$\lambda_i = (\pr_i \circ \beta)^*\lambda_{X(1)}$, where 
$\beta: \HMSoodual \to X(1)_d^2$ is the projection
and $\lambda$ is the Hodge class on $X(1)_d$. There, we also defined the
boundary curves $R^{(i)}$, that obey the relation  $R^{(i)} = \frac{12}{d} \lambda_i.$
\par
In Corollary~\ref{cor:boundarycomps} we gave a description of the boundary 
with two components $D^{(i)}$,  mapping surjectively to $R^{(i)}$ respectively for $i=1,2$.
The discussion in Section~\ref{sec:quotientstack} implies that
$\pi^*R^{(i)}=D^{(i)}$. In particular, we have the relation
\begin{align}\label{eqn:boundary-upstairs-downstairs}
D^{(i)} = \pi^* R^{(i)} = \frac{12}{d} \pi^*\lambda_i
\end{align}
\par
For $i=1,2$ let $N^{(i)}_{\Box}$ be the pullback of the zero section 
$N_{X(d)}$ of the compactified universal family  $\Familyelldleveld$ 
of elliptic curves via the $i$-th projection  to $\Familyelldleveld^2$.
We denote by 
\be
N^{(i)} = \tilde\tau(N^{(i)}_{\Box})
\ee
the image of these zero sections in $\Familyoodual$. Note that\ 
$N^{(i)} = \tfrac1{\Delta_d}\tilde\tau_*(N^{(i)}_{\Box})$.
\par
With the same letter and the additional subscript $\ell \ttors$ we denote the
corresponding divisors of the multi-section of primitive $\ell$-torsion points 
on the family over  $X(d)$, over $X(d)^2$ and over $\HMSoodual$ respectively.
Their classes are related as follows.
\par
\begin{prop} \label{prop:Ntor1tor}
In $\CH^1(\Familyoodual)$, we have for $\ell>1$
 \[N^{(i)}_{\ell \ttors} = \tfrac{\Delta_\ell}{\ell} (N^{(i)} + \pi^*\lambda_i).\]
\end{prop}
\par
\begin{proof} All the quantities involved are pull backs from the
universal family $E(1)$ (we calculate in $\Pic_\QQ$ of a quotient stack)
over $X(1)$ and we prove the relation there. The rational Picard group
of an elliptic fibration is generated by the zero section $N$, the
class $F$ of a fiber and the components of the singular fibers, with the
relation that the sum of all the components are equal to a smooth fiber. Since
all the singular fibers are irreducible here  we can disregard the 
singular fibers.
\par
Consequently, we write  $N_{\ell \ttors} = aN + bF$. Intersecting
with another fiber shows that $a = \tfrac{\Delta_\ell}{\ell}$. Intersecting
with $N$ shows that $b=-aN^2 = a\deg(\lambda)$ (\cite[Eq.~(12.6)]{Kd63}). 
Since the fiber
classes are pulled back from $X(1)$, where any two points are linearly
equivalent, we may write $bF = \tfrac{\Delta_\ell}{\ell} \bar{\varpi}^*\lambda$, 
where $\bar{\varpi}:E(1) \to X(1)$ is the map of the universal family.
\end{proof}
\par

\subsection{Elliptic Jacobi forms}

In this section we recall the classical theory of elliptic Jacobi
forms for $\tilde\Gamma(1)$ (see e.g. \cite{ezJacobi}), specify
a bundle they are sections of and use this to determine the class
of the divisor where the Jacobi form vanishes. Our method
follows \cite{KramerJac91}, but we redo this case as preparation
for the case of Hilbert Jacobi forms in the next section, to include
non-integral weight and index as well as non-cusp forms, and clarify
the imprecise statement in \cite[Proposition~2.4]{KramerJac91}.
\par
We start with the standard definition and explain the notations afterwards.
\begin{defn} An {\em elliptic Jacobi form} of weight $\kappa \in \tfrac12 \ZZ$ 
 and index $m \in \tfrac1{2d}\ZZ$
 for the group $\wt\Gamma(d)_d = \Gamma(d)_d\semidirect(\ZZ\oplus d\ZZ)$
 and the multiplier $\chi$
 is a holomorphic function $f:\HH \times\CC \to \CC$ such that
\begin{enumerate}[(i)]
 \item $f\barop{\kappa,m}{M,r}(z,u) = \chi(M,r)f(z,u)$ for all $(M,r)\in \wt\Gamma(d)_d$.
 \item For each cusp $C$ with $M_C$, $q_C$ and $\zeta_C$ as defined
in Section~\ref{subsec:fan}, $f$ has a Fourier development
\begin{align*}
f(z,u) \cdot j_{\kappa,m}(M_C^{-1},z,u)^{-1} &= \sum_{0 \leq s \in \ZZ}\ \sum_{t\in \ZZ} 
    c_{C,s,t}\, q_{C}^{s}\zeta_{C}^{t}
\end{align*}
for some $c_{C,s,t} \in \CC$, which vanish unless $4sm - t^2 \geq 0$.
\end{enumerate}
\end{defn}
\par
The divisor $\div f$ of a Jacobi form is well-defined as a subset of $\oFamilyelldleveld$, 
since the exponential factors in the transformation rule (see \eqref{eqn:baroperator-elliptic})
do not change the vanishing order of the function. However, 
$\div f$ does not define a class in $\Pic(\Familyelldleveld)$, since the boundary contribution
is not well-defined. Later (compare Theorem~\ref{thm:howtointersect}) we are interested
in the class of the topological closure $\overline{\div f}$ in $\Pic(\Familyelldleveld)$.
This class however is not determined by the parameters (weight, index, multiplier)
of the Jacobi form, as one can easily see already for modular forms. We
will talk about divisor classes once we introduced the bundle of Jacobi forms.
\par
Note that condition ii) is for historical reasons only. It holds for
the most important examples (theta functions introduced below, and also 
Fourier-Jacobi coefficients of Siegel modular forms) and guarantees the
finite-dimensionality of the space of Jacobi forms for fixed parameters.
However, many other (cone) conditions would do as well and 
fixing the bundle $\Jacforms_{\kappa,m}(\Familyelldleveld)$ is independent of
this choice.
\par
\medskip
\paragraph{\bf The slash operator for $\wt\Gamma(d)_d$}
In order to define the slash operator we let
\begin{align*}
  j_{\kappa,m}(\gamma, z, u) &= (cz+e)^{-\kappa} \cdot \e\biggl(-m\, \frac{c(u+r_1z+r_2)^2}{cz+e}\biggr) \, \cdot \, \e(m(r_1^2z + 2r_1u)).
 \end{align*}
For $\kappa$ integral, the function $j_{\kappa,m}$ is an {\em automorphy factor} 
for $\gamma\in \tilde\Gamma(1)$ called 
classical automorphy factor, i.e.
\begin{align*}
j_{\kappa,m}(\gamma_1\gamma_2, z,u) = j_{\kappa,m}(\gamma_1,\gamma_2(z,u)) \cdot j_{\kappa,m}(\gamma_2,z,u). 
\end{align*}
and we define 
\begin{align}\label{eqn:baroperator-elliptic}
 f\barop{\kappa,m}{\gamma}(z,u) := f(\gamma(z,u)) \cdot j_{\kappa,m}(\gamma,z,u).
\end{align}
\par
In this case $\chi:\wt\Gamma(d)_d \to \CC^\times$ is just an abelian character.
For general $\kappa$, the map $\chi$ is a {\em multiplier}, i.e.\ a map so
that $j_{\kappa,m}(\cdot) \chi^{-1}(\cdot)$ is an automorphy factor for a fixed
choice of the determination of $(cz+e)^{-\kappa}$. In any case, $\chi$ is
supposed to be finite, i.e. $\chi^M =1$ for some $M\in\NN$.
\par
Let $d\geq 3$, $\ell$ be integers. Recall that
\[\tilde\Gamma(\ell)_d = \Gamma(\ell)_d\semidirect (\tfrac{\ell}{d}\ZZ\oplus\ell\ZZ) = 
 \diag(d,1)\cdot \Gamma(\ell)\semidirect \ell\ZZ^2 \cdot \diag(d^{-1},1).\]
\par
\begin{lemma}
 For $\kappa\in \ZZ$, $m\in \ZZ$, the function $j_{\kappa,md/\ell^2}$
 is an automorphy factor for the twisted group $\tilde\Gamma(\ell)_d$.
\end{lemma}
\begin{proof}
 Consider the map $\varphi: \HH\times\CC \to\HH\times \CC$, $(z,u)\mapsto(dz,\ell u)$. 
 It is equivariant with respect to the map 
$\Phi: \Gamma(\ell)\semidirect\ZZ^2 \to \Gamma(\ell)_d\semidirect(\tfrac{\ell}{d}\ZZ\oplus \ell\ZZ)$
given by 
$$\Bigl(\tmatrix{a}{b}{c}{e},(r_1,r_2)\Bigr) \mapsto
(\tmatrix{a}{bd}{c/d}{e}, (\tfrac{\ell}{d} r_1, \ell r_2))\,.$$
Since by pullback 
\begin{align*}
j_{\kappa,m}\circ (\Phi\times \varphi)^{-1} (\gamma,z,u) &= (cd\tfrac{z}{d} + e)^{-\kappa} 
  \e\bigl(m(\tfrac{d^2}{\ell^2}r_1^2\tfrac{z}{d} + 2\tfrac{d}{\ell}r_1 \tfrac{u}{\ell})\bigr)\\
&\quad\quad \cdot \e\biggl(-m\frac{cd(\tfrac{u}{\ell} + \tfrac{d}{\ell}r_1\tfrac{z}{d} + \tfrac1\ell r_2)^2}{cd\tfrac{z}{d} + e}\biggr)\\
&= j_{\kappa,md/\ell^2}(\gamma,z,u)\,,
\end{align*}
the classical automorphy factor $j$ restricted to  
$\Gamma(\ell)\semidirect\ZZ^2$ with $m\in\ZZ$ and $\kappa\in\ZZ$ is
transformed into an automorphy factor for $\Gamma(\ell)_d\semidirect 
(\tfrac\ell d \ZZ\oplus \ZZ)$.
\end{proof} 
\par
\medskip
\paragraph{\bf A bundle of elliptic Jacobi forms}
It is well-known that an automorphy factor like $j_{\kappa,m}\chi^{-1}$ for a
group like
$\tilde\Gamma(d)_d$ defines a line bundle 
$\Jacforms_{\kappa,m}^\chi(\oFamilyelldleveld)$ 
on $\HH~\times~\CC/\tilde\Gamma(d)_d = \oFamilyelldleveld$. 
We specify an extension of $\Jacforms_{\kappa,m}^\chi(\oFamilyelldleveld)$ to
$\Familyelldleveld$. For simplicity, let us first assume that $j_{\kappa,m}$ is
already an  automorphy factor. We consider the line bundle induced on the open
set $X_C$ introduced in Section~\ref{subsec:fan}; in fact, it suffices to work
over the cusp $\infty$ and carry the arguments over to any other cusp $C$ using
the elements $M_C$.
As the slash operator is trivial on $P_\infty^n(d)_d$, so is the line bundle
induced by $j_{\kappa,m}$ on $X_\infty$. We extend it to a line bundle on
$X_{\infty,\Sigma}$ by declaring on $T_{\sigma_k}$ the Laurent series
\[q_k^{-mk^2}\zeta_k^{-m(k+1)^2} \sum_{i,j\geq 0} c_{i,j} q_k^i \zeta_k^j\]
to be holomorphic. Since by Lemma~\ref{lem:boundary-action-2dims}, $f_k =
q_k^{-mk^2}\zeta_k^{-m(k+1)^2}$ is mapped to $f_k\barop{\kappa,m}{\tilde g} =
f_{k-s_1}\cdot \alpha$ for some $d$-th root of unity $\alpha$ by the element
$\tilde g(s_1,s_2,\varepsilon,b)\in P_\infty(1)_d$, it follows that this
extension descends to a well-defined line bundle on $Y_{\infty,\Sigma}$.
Performing this extension over all cusps, we obtain a well-defined line bundle
$\Jacforms_{\kappa,m}(\Familyelldleveld)$ on $\Familyelldleveld$ that restricts
to $\Jacforms_{\kappa,m}(\oFamilyelldleveld)$ on the open family.
\par
In the presence of a non-trivial multiplier $\chi$, the line bundle induced on
$X_\infty$ may not be trivial. Still it is a local system, which means that the
sections in two trivializations are transformed into each other by
multiplication by a non-zero constant. This entails that we can use the same
definition as above for the extension. Note also that the arguments show in fact
that the extension $\Jacforms^\chi_{\kappa,m}(\Familyelldleveld)$ is a
$\tilde\Gamma(d)_d/\tilde\Gamma(1)_d$-equivariant bundle (as long as the
automorphy factor $j_{\kappa,m}\chi^{-1}$ is well-defined on
$\tilde\Gamma(1)_d$).
\par
In order to make the connection with Jacobi forms, we rewrite the Fourier expansion of a Jacobi form $f$ at the cusp $\infty$ using
\[\zeta_\infty = \zeta_{k}^{k+1}q_{k}^k,\quad q_\infty = \zeta_{k}q_{k}\]
and obtain
\[f(z,u) = \sum_{\substack{s,t \in \ZZ, s\geq 0,\\ 4sm-t^2 \geq 0}}
    c_{s,t}\, q_{\infty}^{s}\zeta_{\infty}^{t} = \sum_{\substack{s,t \in \ZZ, s\geq 0,\\ 4sm-t^2 \geq 0}}
    c_{s,t}\, q_k^{s+kt}\zeta_{k}^{s+(k+1)t}.\]
It is easy to check that the smallest $q_{k}$-exponent appearing is
\[\min\set{s+kt}{s,t\in \ZZ, 4sm-t^2 \geq 0, s\geq 0} \geq -mk^2,\]
and that a similar statement holds for the smallest $\zeta_k$-exponent.
Thus, $f$ is a holomorphic section of the bundle extension
$\Jacforms_{\kappa,m}(\Familyelldleveld)$.
\par
With this choice of extension, the class of $\div(f)$ is well-defined and has
been calculated in \cite[Proposition~2.4]{KramerJac91}. 
The result is not needed in the sequel, but we will follow his method in the next subsections very closely
to prove Theorem~\ref{thm:class_of_HJF}. 
\par

\subsection{Hilbert Jacobi forms} \label{sec:HJforms}
In this section, we define Jacobi forms for the
pseudo-Hilbert modular surfaces analogously to the elliptic case by
an automorphy factor and a condition on the 
Fourier development at the boundary. Then we
describe an extension of the line bundle 
induced by the automorphy factor on $\oFamilyoodual$ to the compactification
$\Familyoodual$, whose global sections will include all Hilbert Jacobi forms.
Again, we first give the well-known definition and explain notation afterwards.
\par
\begin{defn} \label{def:HJF}
 A {\em Hilbert Jacobi form} of weight $\kappa=(\kappa_1,\kappa_2)\in
\tfrac12\ZZ^2$
 and index $m=(m',m'')\in \tfrac12\frakod$
 for the group $\Semidirectoodual$ and multiplier $\chi$
 is a holomorphic function $f:\HH^2\times\CC^2\to \CC$ such that
\begin{enumerate}[(i)]
 \item $f((M,r)(z,u))\cdot \tilde\jmath_{\kappa,m}((M,r),z,u) \ =
\chi(M,r)\, f(z,u)$ for all $(M,r)\in \Semidirectoodual$.
 \item $f$ has Fourier developments
\ba \label{eq:FD_HJF}
f(z,u) &\ = \sum_{s'\in \ZZ}\ \sum_{t'\in \ZZ} 
    c_{s',t'}(z_2,u_2)\, q_{1}^{s'}\zeta_{1}^{t'}\\
    &\ = \sum_{s''\in \ZZ}\ \sum_{t''\in \ZZ} 
    c_{s'',t''}(z_1,u_1)\, q_{2}^{s''}\zeta_{2}^{t''}
\ea
in the local coordinates
\[q_{i} = \e(\tfrac{z_i}{d^2}),\quad \zeta_i = \e(\tfrac{u_i}{d}),\]
where $c_{s',t'}$, $c_{s'',t''}$ are holomorphic functions, which vanish 
unless 
\[4sm - t^2\geq 0 \quad \text{and}\quad s\geq 0.\]
\end{enumerate}
\end{defn}
\par
In this definition, 
\ba\label{eqn:AF-HJF-on-Semidirectoodual} 
\tilde\jmath_{\kappa,m}(\gamma,z,u) &\= \e(\tr_{K/\QQ}(m(r_1^2z + 2r_1u)))\,\, \prod_{i=1}^2(c^{(i)}z_i+e^{(i)})^{-\kappa_i} \cdot \\
&\qquad \quad \cdot\e(-\tr_{K/\QQ}(m(cz+e)^{-1}c(u+z^*r_1^T+r_2^T)^2))\\
\ea
and one checks that for $\kappa$ integral 
the function $(z,u) \mapsto \tilde\jmath_{\kappa,m}(\gamma,z,u)$ is 
an automorphy factor for $\Semidirectoodual$. In the general case,
for  $\kappa$ not necessarily integral, 
a multiplier is defined to be a map $\chi:\Semidirectoodual \to \CC^\times$ 
such that for a fixed determination of
$\tilde\jmath_{\kappa,m}$ the product $\tilde\jmath_{\kappa,m}\chi^{-1}$
is an automorphy factor for $\Semidirectoodual$. We suppose throughout
that $\chi(\gamma)$ has finite order for $\gamma\in \Semidirectoodual$.
We will not need more details, since the multipliers trivialize
after taking tensor powers and so they do not effect a statement
on the rational Picard group as Theorem~\ref{thm:class_of_HJF}.
\par
Note also that
\be\label{eqn:jmath_is_product}
\tilde\jmath_{\kappa,m} = j^{(1)}_{\kappa_1,m'}\cdot j^{(2)}_{\kappa_2,m''}
\ee
where $j^{(i)}_{\kappa_i,m^{(i)}}(\gamma,z,u) =
j_{\kappa_i,m^{(i)}}(\gamma^{(i)}, z_i, u_i)$. 
\par
\medskip
\medskip
\paragraph{\bf A bundle of Hilbert Jacobi forms} 
We denote by  $\Jacforms_{\kappa,m}^\chi(\oFamilyoodual)$
 the line bundle defined by the automorphy factor
 $\tilde\jmath_{\kappa,m} \chi^{-1}$ 
on the open variety $\oFamilyoodual$. In order to extend it, we
proceed as in the elliptic case. We work local coordinates near a boundary
divisor, say $D^{(1)}$ and suppose first that $\chi=1$. The local coordinates
are given by Lemma~\ref{lem:localcoords_at_boundaries_of_A} by
\[\zeta_{1,k}, q_{1,k}, z_2, u_2,\]
and the line bundle induced by $\tilde\jmath_{\kappa,m}$ is trivial. Again, we
declare sections to be holomorphic if they are of the form
\be \label{eqn:HJF_bundel}
q_{1,k}^{-m' k^2} \zeta_{1,k}^{-m' (k+1)^2} \cdot f
\ee
for a holomorphic function $f = f(\zeta_{1,k},q_{1,k},z_2,u_2)$. For
a non-trivial multiplier $\chi$, we have to pass to local systems, but this
definition still makes sense, since it is independent of the chosen
trivialization of the local system.
\par
Alternatively, we can construct the extension (for $\chi=1$) by using
\eqref{eqn:jmath_is_product}, which translates into
\[\tilde\tau^*\Jacforms_{\kappa,m}(\oFamilyoodual) \isom
\pr_1^*\Jacforms_{\kappa_1,m'}(\oFamilyelldleveld) \tensor
\pr_2^*\Jacforms_{\kappa_2,m''}(\oFamilyelldleveld).\]
and the fact that the latter bundle has an extension, which is in fact
$H_{d^2}$-equivariant and thus induces a bundle on the quotient. (Note that for
$m\in d\ZZ^2$, it is even
$\tilde\Gamma(1)_d^2/\tilde\Gamma(d)_d^2$-equivariant, but for general rational
index $m$, $j$ is not an automorphy factor for $\tilde\Gamma(1)_d$.)            
\par
From the Fourier development~\eqref{eq:FD_HJF}
and the coordinate transformations~\eqref{eqn:zeta_ik_q_ik_to_zeta_q}
we deduce that a Hilbert Jacobi form has near the boundary divisors
$D^{(1)}_{\infty,k}$ given by $q_{1,k}=0$ a Fourier development
\be \label{eq:FJ_1stBD}
f(z,u) = \sum_{\substack{s',t':\\ 4s'm_1 -t'^2\geq 0}} c_{s',t'} \zeta_{1,k}^{s' + (k+1)t'}q_{1,k}^{s' + kt'}.
\ee
The same estimate as for elliptic Jacobi forms yields that Hilbert Jacobi forms
are indeed holomorphic sections of $\Jacforms^\chi_{\kappa,m}(\Familyoodual)$.
\par

\subsection{Theta functions} \label{sec:HilbTheta}
We recall the definition of the 
classical (Siegel) theta-functions.
We use the convention that $x = (x_1,x_2)$ and  $\gamma_i$ are row vectors
while and $v = (v_1,v_2)^T$ is a column vector. Let 
\ba
\theta \Tchi{\gamma_1}{\gamma_2}: 
\begin{cases}
\HH_g\times \CC^g & \to\quad \CC \\ 
(Z, v)  &\mapsto \quad 
\displaystyle \sum_{x \in \ZZ^g + \tfrac12\gamma_1} \e\left(\tfrac{1}{2}xZx^T +
x(v+\tfrac12\gamma_2^T)
\right). 
\end{cases}
\ea
be the {\em theta function with half-integral characteristic} $\gamma =
(\gamma_1,\gamma_2)
\in \ZZ^2$. The evaluation of a theta-function 
at $v=0$ is called {\em theta constant}. The theta-function (and the
characteristic $(\gamma_1,\gamma_2)$) is called {\em odd} if 
$\gamma_1\gamma_2^T$ is odd and {\em even} otherwise.
Odd theta constants vanish identically as functions in $Z$.
The theta constants are modular forms of weight $1/2$ for the
subgroup $\Gamma(4,8)$ of $\Sp(2g,\ZZ)$, non-zero if and only
if $(\gamma_1,\gamma_2)$ is even.
\par
For a matrix $M = \tmatrix{A}{B}{C}{E} \in \Sp(2g,\ZZ)$ and a 
vector $\lambda = (\lambda_1,\lambda_2) \in \ZZ^{2g}$
the theta function transforms (see \cite{bl}) as
\begin{align}
\begin{split}\label{eqn:Siegel-theta-I}
\theta \Tchi{(M\gamma)_1}{(M\gamma)_2}& (M(Z), (CZ+E)^{-T}v) = \\
& \theta \Tchi{\gamma_1}{\gamma_2}(Z, v) \cdot 
\chi_{\theta}(M)\cdot \det(CZ+E)^{1/2} \e(\tfrac1{2} v^T(CZ+E)^{-1}Cv),
\end{split}
\\
\begin{split}\label{eqn:Siegel-theta-II}
\theta \Tchi{\gamma_1}{\gamma_2}& (Z, v + Z\lambda_1^T + \lambda_2^T) =\\
& \theta\Tchi{\gamma_1}{\gamma_2}(Z, v) 
  \cdot \e(\tfrac{\gamma_1}2\lambda_2^T - \tfrac{\gamma_2}2\lambda_1^T -
\tfrac1{2}\lambda_1Z\lambda_1^T - v^T\lambda_1^T)\, .
\end{split}
\end{align}
Here, $\chi_\theta$ is a multiplier, which takes values in the $8$-th roots of
unity, 
and $M$ acts on the characteristic by
\begin{align*}
 (M\gamma)_1 &= E\gamma_1^T - C\gamma_2^T + (CE^T)_0\\
 (M\gamma)_2 &= -B\gamma_1^T + A\gamma_2^T + (AB^T)_0,
\end{align*}
where $(S)_0 = (s_{11},\dots,s_{gg})$ denotes the diagonal vector of a matrix
$S\in \RR^{g\times g}$.
\par
\medskip
We are interested in {\em Hilbert theta functions} 
(with half-integral characteristics), 
the pullback of the Siegel theta-function for $g=2$ 
to $\HH^2\times \CC^2$ via the modular
embedding $\tilde\psi$ defined in Section~\ref{subsec:modular_embeddings}.
Concretely, these theta functions are given as the power series
\ba
\vartheta\Tchi{\tilde \gamma_1}{\tilde \gamma_2}(z,u) :=
\psi^*\theta\Tchi{\gamma_1}{\gamma_2}(z,u) &= \sum_{x \in
\ZZ^2+\tfrac{\gamma_1}2}\e(\tfrac1{2}x Az^*A^Tx^T + x(Au + \tfrac12\gamma_2^T)\\
	&= \sum_{x\in (\ZZ^2 + \tfrac{\gamma_1}2)A} \e(\tfrac1{2}xz^*x^T + x(u +
\tfrac1{2}B\gamma_2^T))\\
	&= \sum_{x\in \frako_{d^2}^\dual + \tfrac{\tilde\gamma_1}2} \e\bigl(
\tr_{K/\QQ}(\tfrac1{2}(x^2z + 2x(u + \tfrac1{2}\tilde\gamma_2)))\bigr)
\ea
where $\tilde \gamma_1 = \gamma_1A \in 
\frako_{d^2}^\dual$, and
$\tilde\gamma_2= \gamma_2B^T \in 
\frako_{d^2}$. We first analyze
the action of $\SLoodual$ on characteristics.
\par
\begin{lemma}
The set of even theta characteristics
decomposes under the action of  $\SLoodual$ into two orbits
\bas
E_0 &\=  \Bigl\{\Tchi{(0,0)}{(0,0)}, 
\Tchi{(1, 0)}{(0, 1)}, \Tchi{(1, 0)}{(0, 0)}, 
\Tchi{(0, 0)}{(0, 1)} \Bigr\}  \\
E_2 &\ = \Bigl\{ \Tchi{(1, 1)}{(0, 0)}, \Tchi{(1, 1)}{(1, 1)}, 
\Tchi{(0, 0)}{(1, 1)}, \Tchi{(0, 1)}{(1, 0)}, \Tchi{(0, 1)}{(0, 0)}, \Tchi{(0, 0)}{(1, 0)} \Bigr\} \\
\eas 
for $d$ even
and into $O_3 = \Bigl\{\Tchi{(0,1)}{(1,0)}\Bigr\}$  and
$$ O_1 =  \Bigl\{ \Tchi{0,0}{0,0}, 
\Tchi{1,0}{0,0},\Tchi{0,0}{1,0}, \Tchi{0,1}{0,0}, 
\Tchi{0,0}{0,1},  \Tchi{1,0}{0,1},
\Tchi{1,1}{0,0}, \Tchi{0,0}{1,1}, \Tchi{1,1}{1,1} \Bigr\}
$$
for $d$ odd.
\end{lemma}
\par
The labeling of the orbits is consistent with the notation for spin 
structures for the reducible
locus, as we will see in Section~\ref{sec:redlocus}. The odd theta 
characteristics form two orbits for $d$ odd and one orbit for $d$ even, 
but we will not need this fact.
\par
\begin{proof}
Recall that in $g=2$ an even theta characteristic can be written as a sum of
three (out of six) odd theta characteristics, and that this representation
is unique up to passing to the complementary triple. Odd theta characteristics
correspond to Weierstra\ss\ points and they have been normalized in 
Proposition~\ref{prop:kani_norm} globally, i.e.\ in a way that is invariant 
under  $\SLoodual$. For $d$ odd the alternating sum of the three Weierstrass\
points
in one fiber is the distinguished even theta characteristic. For $d$ even there
are two kinds of triples: four triples (and their complements) can be formed
by picking one Weierstra\ss\ point out of each pair from Proposition~\ref{prop:kani_norm}.
Six triples (and their complements) can be formed
by picking both Weierstra\ss\ point from such a pair and a third point. These
correspond to the orbits $E_0$ and $E_2$ respectively.
\par
It is easy to show that these orbits do not decompose further by exhibiting
appropriate elements of $\SLoodual$ and the transformations 
 \begin{align*}
  \tilde\gamma_1 &\mapsto \widetilde{(M\tilde\gamma)}_1 = 
    \tilde\gamma_1 e^* - \tilde\gamma_2c^*  + (B^Tc^*e^*B)_0^TA\\
  \tilde\gamma_2 &\mapsto \widetilde{(M\tilde\gamma)}_2 = 
    -\tilde\gamma_1b^* + \tilde\gamma_2a^* + (Aa^*b^*A^T)_0^TB^T,
 \end{align*}
where $M = \tmatrix{a}{b}{c}{e}\in \Semidirectoodual$ that
follow from~\eqref{eqn:Siegel-theta-I} and the definition of the modular embedding.
\end{proof}
\par
\begin{prop} \label{prop:thetaisHJF}
The Hilbert theta functions are Hilbert Jacobi forms of weight
$(\tfrac12,\tfrac12)$ and index $(\tfrac12,\tfrac12)$
for some subgroup of finite index in $\Semidirectoodual$.
\par
For $d$ odd, one of the Hilbert theta functions is a
Hilbert Jacobi form for the full group  $\Semidirectoodual$.
With our choice of $B$ and the modular embedding, this is
$\vartheta\Tchi{(0,1)}{(1,0)}$.
\end{prop} 
\par
\begin{proof}
The group $\Semidirectoodual$ acts on $\vartheta$ by
\begin{align}
 \begin{split}
 \vartheta\Tchi{\tilde\gamma_1}{\tilde\gamma_2}(z,u) &=
\vartheta\Tchi{(M\tilde\gamma)_1}{(M\tilde\gamma)_2}(M(z), (c^*z + e^*)^{-1}(u +
z^*r_1^T + r_2^T))\\
    &\cdot \prod_{i=1}^2(c^{(i)}z_i + e^{(i)})^{-1/2}\cdot
\e\Bigl(\tfrac12\tr_{K/\QQ}(r_1^2z + 2r_1u)\Bigr)\\
    &\cdot \e\Bigl(-\tfrac12\tr_{K/\QQ}((u + z^*r_1^T +
r_2^T)^T{(cz+e)}^{-1}{c(u + z^*r_1^T + r_2^T)})\Bigr)\\
    &\cdot \chi_\theta(\Psi(M))^{-1} \cdot \e\bigl(\tr_{K/\QQ}
(\tfrac{\tilde\gamma_1}2 r_2 - \tfrac{\tilde\gamma_2}2 r_1)\bigr)
 \end{split}
\end{align}
where $(M,r) = \bigr(\tmatrix{a}{b}{c}{e},(r_1,r_2)\bigl)\in \Semidirectoodual$.
This proves the claim on the weight and the index. The second statement
follows from the previous lemma.
\end{proof}
\par
\par
\medskip
Last, we list the theta characteristics and their images under the
transformation $\tilde \gamma_1 = \gamma_1A$, respectively
$\tilde\gamma_2= \gamma_2B^T$. The first row is multiplied by $d$
for convenience.
\begin{table}[ht]
\begin{center}
\begin{tabular}[b]{c|lllll}
  $\Tchi{\gamma_1}{\gamma_2}$  & $\Tchi{0,0}{0,0}$ & 
  $\Tchi{1,0}{0,0}$ & $\Tchi{0,0}{1,0}$ & $\Tchi{0,1}{0,0}$ & 
  $\Tchi{0,0}{0,1}$ \\[1ex]
  \hline\\[-1.5ex]
 $\Tchi{d\tilde\gamma_1}{\tilde\gamma_2}$ & $\Tchi{0,0}{0,0}$ & 
  $\Tchi{d,0}{0,0}$ & $\Tchi{0,0}{1,1}$ & $\Tchi{-1,1}{0,0}$ & 
  $\Tchi{0,0}{0,d}$ \\[1ex]
\\
  $\Tchi{\gamma_1}{\gamma_2}$  & $\Tchi{1,0}{0,1}$ & $\Tchi{0,1}{1,0}$ &
  $\Tchi{1,1}{0,0}$ & $\Tchi{0,0}{1,1}$ & $\Tchi{1,1}{1,1}$\\[1ex]
  \hline\\[-1.5ex]
 $\Tchi{d\tilde\gamma_1}{\tilde\gamma_2}$ & 
 $\Tchi{d,0}{0,d}$ & $\Tchi{-1,1}{1,1}$ &
  $\Tchi{d-1,1}{0,0}$ & $\Tchi{0,0}{1,d+1}$ & $\Tchi{d-1,1}{1,d+1}$\\[1ex]
\end{tabular}
\end{center}
\caption{Even theta characteristics under base change}
\label{table:eventhetachars}
\end{table}

\begin{table}[ht]
\begin{center}
\begin{tabular}[c]{c|llllll}
  $\Tchi{\gamma_1}{\gamma_2}$  & $\Tchi{1,0}{1,0}$ & 
  $\Tchi{1,1}{1,0}$ & $\Tchi{1,0}{1,1}$ & $\Tchi{0,1}{0,1}$ & 
  $\Tchi{1,1}{0,1}$ & $\Tchi{0,1}{1,1}$\\[1ex]
\hline\\[-1.5ex]
 $\Tchi{d\tilde\gamma_1}{\tilde\gamma_2}$ & $\Tchi{d,0}{1,1}$ & 
  $\Tchi{d-1,1}{1,1}$ & $\Tchi{d,0}{1,d+1}$ & $\Tchi{-1,1}{0,d}$ & 
  $\Tchi{d-1,1}{0,d}$ & $\Tchi{-1,1}{1,d+1}$\\[1ex]
\end{tabular}
\end{center}
\caption{Odd theta characteristics under base change}
\label{table:oddthetachars}
\end{table}



\subsection{The divisor of a Hilbert Jacobi form}
In this section, we determine the class of the bundle of Hilbert Jacobi forms
in terms of the pullbacks of the Hodge bundles $\pi^*\lambda_i$ and the zero
sections $N^{(i)}$, that is we complete the proof of 
Theorem~\ref{thm:class_of_HJF}.
\par
The plan is to reduce the weight and index of any Hilbert Jacobi form
to zero with the help of the following two functions, whose divisor
class we can compute.
\par
\begin{lemma} \label{le:divvartheta}
The function $\vartheta^{(i)}_d\Tchi{1}{1}:\HH^2\times\CC^2\to \CC$, given by
\[\vartheta^{(i)}_d\Tchi{1}{1}(z,u) = \sum_{x\in \ZZ}
\e\bigl(\tfrac12(x+\tfrac12)^2\tfrac{z_i}{d} +
(x+\tfrac12)(u_i+\tfrac12)\bigr)\]
as a pullback of a one-variable theta function, is a Hilbert Jacobi 
form for $\Semidirectoodual$ of weight $\kappa$ with 
$\kappa_j=  \tfrac12\delta_{ij}$ and
index $(m^{(1)}, m^{(2)})$  where $ m^{(j)}=\tfrac{d}{2}\delta_{ij}$. Its divisor
is 
\[\div\vartheta^{(i)}_d\Tchi{1}{1} = N^{(i)} + \frac{d}{8}D^{(i)}.\]
\end{lemma}
\par
\begin{proof} One immediately deduces from the theta transformation
formula that 
\begin{align*}
\vartheta^{(i)}_d\Tchi{1}{1}((M,r)(z,u))
\cdot j^{(i)}_{\tfrac12,\tfrac{d}2}((M,r),z,u)
 = \e(\tfrac12r_2^{(i)} - \tfrac12 dr_1^{(i)}) \cdot 
\chi^{(i)}_\theta(M)\cdot
\vartheta^{(i)}_d\Tchi{1}{1}(z,u), 
\end{align*}
for $(M,r)\in \Semidirectoodual$, where $\chi^{(i)}_\theta(M) :=
\chi_\theta(\diag(d^{-1},1)M^{(i)}\diag(d,1))$, and where $\chi_\theta$ denotes
the multiplier introduced in the $1$-dimensional theta transformation
formula~\eqref{eqn:Siegel-theta-II}.
\par

%
For the divisor calculation we may focus on the case $i=1$.
At the boundary divisor $D^{(1)}$, which in the local coordinates
$(\zeta_{1,k},q_{1,k},z_2,u_2)$ of
Lemma~\ref{lem:localcoords_at_boundaries_of_A} is given by $q_{1,k} = 0$, we
have the Fourier development
\begin{align*}
 \vartheta_d^{(1)}\Tchi{1}{1} &= \sum_{x\in\ZZ} q_1^{d/2(x+1/2)^2}
\zeta_1^{d(x+1/2)} \cdot \e(\tfrac12(x+\tfrac12))\\
    &= \sum_{x\in \ZZ} q_{1,k}^{d/2(x+1/2)^2 +
kd(x+1/2)}\zeta_{1,k}^{d/2(x+1/2)^2 + (k+1)d(x + 1/2)} \cdot
\e(\tfrac12(x+\tfrac12))
\end{align*}
 Thus, the vanishing order of $\vartheta_d^{(1)}\Tchi{1}{1}$ at $q_{1,k}=0$ as
\textit{a function} is given by
\begin{align*}
 \min_{x\in \ZZ} \tfrac{d}{2}(x+\tfrac12)^2 + kd(x+\tfrac12) &=
\tfrac{d}{2}\bigl(\min_{x\in \ZZ} x^2 + (1+2k)x + \tfrac14 + k\bigr)\\
    &= \tfrac{d}{2}\bigl(\min_{x\in\ZZ} (x + \tfrac12 + k)^2 - (\tfrac12 + k)^2
+ \tfrac14 + k\bigr)\\
    &= \tfrac{d}{2}\bigl(\min_{x\in \ZZ} (x + \tfrac12 + k)^2 - k^2\bigr)\\
    &= \tfrac{d}{8} - k^2\tfrac{d}{2}
\end{align*}
Using \eqref{eqn:HJF_bundel}, we see that the vanishing order as \textit{a
section} of the bundle of Hilbert Jacobi forms is $\tfrac{d}8$. Thus,
\[C = \div \vartheta^{(1)}_d\Tchi{1}{1} - \tfrac{d}{8}D^{(1)}\]
is a divisor on $\Familyoodual$, whose support is disjoint from
the boundary.
\par
The divisor of the classical theta function $\theta\Tchi{1}{1}$ 
on $\oFamilyelld = \HH\times\CC/(\Gamma(d) \semidirect(d\ZZ)^2)$ is equal
to $d^2$-times the zero-section. This relation persists under
passing to the quotient by the conjugate group $\tilde\Gamma(d)_d$ via
the equivariant isomorphism $(z,u)\mapsto (\tfrac{z}{d},u)$. Thus
\[\cO_{(\oFamilyelldleveld)^2}(\div \vartheta_d^{(1)}\Tchi{1}{1}) 
  \isom \cO_{(\oFamilyelldleveld)^2}(d^2N_{X(d)^2}^{(1)}).\]
Therefore,
\begin{align*}
 \deg(\tilde\tau) C &= \tilde\tau_*\tilde\tau^*C\\ 
    &= \tilde\tau_*\cO_{(\oFamilyelldleveld)^2}(\div
\vartheta_d^{(1)}\Tchi{1}{1}) -
\deg(\tilde\tau)\tfrac{d}{8}D^{(1)}\\
    &= d^2\tilde\tau_*N_{X(d)^2}^{(1)} - \deg(\tilde\tau)\tfrac{d}{8}D^{(1)}\\
    &= d^2{\Delta_d}N^{(1)} - \deg(\tilde\tau)\tfrac{d}{8}D^{(1)}\, ,
\end{align*}
which together with $\deg(\tilde\tau) =d^2\Delta_d$ implies the claim.
\end{proof}
\par
\begin{lemma}
The pullback of the one-variable $\eta$-function $\eta^{(i)}:\HH^2\times\CC^2\to \CC$, given by
\begin{align*}
 \eta^{(i)}(z,u) &= \e(\tfrac{z_i}{24d})\prod_{n=1}^\infty
\bigl(1-\e(n\tfrac{z_i}{d})\bigr),
\end{align*}
is a Hilbert Jacobi form for $\Semidirectoodual$ of weight 
$(\kappa_1,\kappa_2)$, where $\kappa_j= \tfrac12\delta_{ij}$,
and index $(0,0)$ with divisor
 \[\div\eta^{(i)} = \frac{d}{24}D^{(i)}.\]
\end{lemma}
\par
\begin{proof} From the well-known one-dimensional transformation
formula one deduces 
\begin{align*}
 \eta^{(i)}\barop{\kappa,0}{M,r} &= \chi_\eta(M^{(i)}) \cdot \eta^{(i)}
\end{align*}
where the multiplier $\chi_\eta$ takes values in the
$24$-th roots of unity.
At $D^{(i)}_{\infty,k}$, the function $\eta^{(i)}$ can be written as
\begin{align*}
  \eta^{(i)} &= q_{i}^{d/24}\prod_{n=1}^\infty (1-q_i^{dn}) 
    = q_{i,k}^{d/24}\zeta_{i,k}^{d/24}\prod_{n=1}^\infty
(1-q_{i,k}^{dn}\zeta_{i,k}^{dn}).
\end{align*}
and the rightmost term does not vanish at $q_{i,k}=0$.
\end{proof}

\begin{proof}[Proof of Theorem~\ref{thm:class_of_HJF}]
Let $f$ be a Hilbert-Jacobi form of weight $\kappa = (\kappa_1,\kappa_2)$ and
index $m = (m',m'')$. Let $g^{(i)}$, $i=1,2$ be the pullback
via $\pr_i$ of a modular form form of weight $24d\ell\kappa_i$ for
$\Gamma(1)_d$.
The function
\begin{align*}
 \Bigl(
  \bigl(\vartheta^{(1)}_d\Tchi{1}{1}\bigr)^{2m'} 
  \bigl(\vartheta^{(2)}_d\Tchi{1}{1}\bigr)^{2m''}
  (\eta^{(1)})^{-2m'} 
  (\eta^{(2)})^{-2m''}\Bigr)^{24 \ell} 
  g^{(1)} g^{(2)} \cdot f^{-24 d \ell}
\end{align*}
has trivial automorphic factor. Hence, it descends 
to a meromorphic function on $\oFamilyoodual$, and
one checks that its extension to $\Familyoodual$ is also meromorphic.
Therefore, we can obtain an explicit divisor linear equivalent 
to $f$ by computing the divisors of the different
factors of the product.
\par
Using the above lemmas, we have
\begin{align*}
 \div f &\sim \frac1d\biggl(2m'(N^{(1)}+ \tfrac{d}{8}D^{(1)}) + 
		           2m''(N^{(2)} + \tfrac{d}{8}D^{(2)}) 
		         - 2m'\tfrac{d}{24}D^{(1)} 
		         - 2m''\tfrac{d}{24}D^{(2)}\\
        &\quad\quad + d\kappa_1\pi^*\lambda_1 + d\kappa_2\pi^*\lambda_2\biggr)\\
        &= \kappa_1\pi^*\lambda_1 + \kappa_2\pi^*\lambda_2
           + \tfrac{2m'}{d}N^{(1)} + \tfrac{2m''}{d}N^{(2)}
           + \tfrac{m'}{6}D^{(1)} + \tfrac{m''}{6}D^{(2)}\, .
\end{align*}
Applying $D^{(i)} = \tfrac{12}{d}\pi^*\lambda_i$ yields the claim.
\end{proof}

\section{The reducible locus} \label{sec:redlocus}

Let $P_{d^2}^\circ \subset \oHMSoodual$ be the {\em reducible locus}, i.e.\  
the locus of points corresponding to abelian surfaces that are isomorphic 
to a product of elliptic curves.
\par
\begin{prop} \label{prop:class_red}
The closure $P_{d^2}$ of the reducible locus has the divisor class 
\begin{align*}
[P_{d^2}] &\= (5 - \tfrac6d)(\lambda_1 + \lambda_2)\,.
\end{align*}
in $\CH^1(\HMSoodual)$. If $d\congruent 1\bmod 2$, its spin components have 
the divisor classes
\begin{align*}
 [P_{d^2,\spin = 3}] &\= (\tfrac12 - \tfrac{3}{2d})(\lambda_1+\lambda_2) \\
 [P_{d^2,\spin = 1}] &\= (\tfrac92 - \tfrac{9}{2d})(\lambda_1+\lambda_2)\,.
\end{align*}
\par
If $d\congruent 0\bmod 2$, its spin components have the divisor classes 
\begin{align*}
 [P_{d^2,\spin = 0}] &\= (2 - \tfrac{6}{d})(\lambda_1+\lambda_2) \\
 [P_{d^2,\spin = 2}] &\= 3(\lambda_1+\lambda_2)\,.
\end{align*}
\end{prop}
\medskip
\begin{cor}\label{cor:eulerchar-reducible}
The spin components of the reducible locus have  Euler characteristic
 \begin{align*}
  \chi(P_{d^2,\spin =3}^\circ) &\=  -\tfrac{1}{288}(d-3)\tfrac{\Delta_d}{d}\\
  \chi(P_{d^2,\spin =1}^\circ) &\= -\tfrac{1}{32}(d-1)\tfrac{\Delta_d}{d}
 \end{align*}
if $d$ is odd, and 
 \begin{align*}
  \chi(P_{d^2,\spin =0}^\circ) &\= -\tfrac{1}{72}(d-3)\tfrac{\Delta_d}{d}\\
  \chi(P_{d^2,\spin =2}^\circ) &\= -\tfrac{1}{48}\Delta_d 
 \end{align*}
if $d>2$ is even.
\end{cor}
\par
This fits with the total count $\chi(P_{d^2}^\circ) =  -\tfrac{1}{144}(5d-6)\tfrac{\Delta_d}{d}$ obtained by several authors, see e.g.\ \cite[Formula~(2.23)]{bainbridge07}.
\par
Given a theta function with characteristic, we write
\[\vartheta_0\Tchi{\tilde\gamma_1}{\tilde\gamma_2}(z) =
\vartheta\Tchi{\tilde\gamma_1}{\tilde\gamma_2}(z,0)\]
for the corresponding {\em theta constant}. 
Mumford shows (\cite[\S~8]{mumford83}) that the reducible locus 
is cut out by the product of all even theta constants and this
product vanishes to order one there.
\par 
If $d\congruent 1\bmod 2$, we define
\[
  \vartheta_{0,\spin = 3} = \vartheta_0\Tchi{0,1}{1,0}\qquad \text{and}\qquad 
  \vartheta_{0,\spin = 1} = \prod_{\Tchi{\gamma_1}{\gamma_2}\in O_1}
\vartheta_0\Tchi{\gamma_1}{\gamma_2}.
\]
These functions are, by the description of the action of  $\SLoodual$ on characteristics
in Section~\ref{sec:HilbTheta}, modular forms for the full group
$\SLoodual$ of weight $(\tfrac12,\tfrac12)$, respectively
of weight $(\tfrac92,\tfrac92)$. If $d\congruent 0 \bmod 2$, define
\[
  \vartheta_{0,\spin = 0} = \prod_{\Tchi{\gamma_1}{\gamma_2} \in E_0} \vartheta_0\Tchi{\gamma_1}{\gamma_2}\qquad \text{and}\qquad 
  \vartheta_{0,\spin  =2} = \prod_{\Tchi{\gamma_1}{\gamma_2} \in E_2} \vartheta_0\Tchi{\gamma_1}{\gamma_2}
\]
Again by the calculations in Section~\ref{sec:HilbTheta} these four functions 
are Hilbert modular forms of weight $(2,2)$ in the first case and
$(3,3)$ in the second. The zero loci of these modular forms
correspond to the spin components of the reducible locus.
\par
\begin{lemma} \label{le:Pd_modularform}
In the open part $\oHMSoodual$ the components of 
the reducible locus are vanishing loci of the modular forms
\[ P^\circ_{d^2,\spin = 3} \= \{\vartheta_{0,\spin = 3}=0\} , \qquad 
\text{respectively}\qquad  P^\circ_{d^2,\spin = 1} \= \{\vartheta_{0,\spin = 1}=0\}\]
for $d$ odd and
\[P^\circ_{d^2,\spin =0} \= \{\vartheta_{0,\spin =0} = 0\} , \qquad 
\text{respectively}\qquad P^0_{d^2,\spin = 2} \= \{\vartheta_{0,\spin = 2} =0\}\]
for $d$ even.
\end{lemma}
\par
\begin{proof} In the case of a smooth genus two curve, 
the function $\vartheta = \vartheta\Tchi{0,0}{0,0}$ vanishes 
at all odd $2$-torsion points, since translating $\vartheta$ by such a point 
gives a theta function with odd characteristic. Consequently, the odd
$2$-torsion points are the Weierstrass\ points. This identification
extends to reducible curves.
\par
A $2$-torsion point $\Tchi{\gamma_1}{\gamma_2}$
is integral, i.e.\ has the same image under the origami map as the node, if
and only if
its base change $\Tchi{d\tilde\gamma_1}{\tilde\gamma_2}$ has $\Tchi{0}{0}$
as first column. So the number of integral Weierstrass\ points in
the vanishing locus of  $\vartheta_0\Tchi{\mu_1}{\mu_2}$ is the
number of odd theta characteristics that have $\Tchi{0}{0}$
as first column after adding $\Tchi{d\tilde\mu_1}{\tilde\mu_2}$.
\par
The claim now follows from inspecting Table~\ref{table:oddthetachars}.
\end{proof}
\par
\begin{lemma}
For $d\congruent 1\bmod 2$, we have
\begin{align*}
  P_{d^2,\spin = 3} &= \tfrac12(\lambda_1+\lambda_2) - \tfrac18 (R^{(1)} +
R^{(2)})\\
  P_{d^2,\spin = 1} &=  \tfrac92(\lambda_2+\lambda_2) - \tfrac38 (R^{(1)} +
R^{(2)})\, . 
 \end{align*}
For $d\congruent 0\bmod 2$, we have: 
\begin{align*}
  P_{d^2,\spin = 0} &= 2(\lambda_1+\lambda_2) - \tfrac12 (R^{(1)} +
R^{(2)})\\
  P_{d^2,\spin = 2} &= 3(\lambda_2+\lambda_2) \, . 
 \end{align*}
\end{lemma}
\begin{proof}
 Let $\vartheta_0\Tchi{\gamma_1}{\gamma_2}$ be an even theta constant. Using the
Fourier development, we have
 \begin{align*}
  \vartheta_0\Tchi{\tilde\gamma_1}{\tilde\gamma_2} &= 
      \e(\tr_{K/\QQ}(\tilde\gamma_1\tilde\gamma_2))
      \sum_{s'\congruent-s''} q_1^{\sfrac12(s'+d\tilde\gamma_1')^2}
q_2^{\sfrac12(s''+d\tilde\gamma_1'')^2} 
	    \e(\tr(s\tfrac{\tilde\gamma_2}{d}))\, .
 \end{align*}
 By symmetry, we may concentrate on the first boundary, which is locally given
 by $q_1 = 0$. The minimal $q_1$-exponent appearing is 
 \[\min_{s'\in\ZZ} \tfrac12(s'+d\tilde\gamma_1')^2 = \begin{cases}
                                                      \tfrac18, &\qquad
\text{if}\ d\tilde\gamma_1' \congruent  1 \bmod 2\\
                                                      0, &\qquad \text{if}\
d\tilde\gamma_1' \congruent 0 \bmod 2
                                                     \end{cases}.\]
Thus, $\vartheta_0\Tchi{\gamma_1}{\gamma_2}$ vanishes at $R^{(i)}$ to
the order $\tfrac18\varepsilon(\gamma)$, where for $\gamma\in \tfrac1{d}\ZZ$,
we set $\varepsilon(\gamma) = 1$, if $d\gamma\congruent 1\bmod 2$
 and $\varepsilon(\gamma) = 0$ else. 
The claim now follows using Table~\ref{table:eventhetachars}. 
\end{proof}
\par
\begin{proof}[Proof of Proposition~\ref{prop:class_red} and 
Corollary~\ref{cor:eulerchar-reducible}]
Proposition~\ref{prop:class_red} follows from the preceding lemmas and 
formula~\eqref{eq:lambdaR}. Since the components of the reducible locus
are all Kobayashi geodesics, the Euler characteristic can be computed
by integration against $\omega_1$. Consequently, 
\begin{equation*}
\begin{aligned}
\chi(P_{d^2,\spin =1}^\circ) &\= \int_{P_{d^2,\spin =1}} -\omega_1 
&&\= -\tfrac12 (\tfrac92 - \tfrac{9}{2d}) \int_{\HMSoodual}  \omega_1\wedge\omega_2 \\
&\= -\tfrac12 (\tfrac92 - \tfrac{9}{2d})\, \chi({\HMSoodual})
&&\= -\tfrac{1}{32}(d-1)\tfrac{\Delta_d}{d}\\
\end{aligned}
\end{equation*}
since $\chi(\HMSoodual) = \tfrac{1}{72}\Delta_d$. The calculation for
the other spin components and for $d$ even is the same. 
\end{proof}


\section{Arithmetic Teichm\"uller curves in
\texorpdfstring{$\omoduli[2]$}{the moduli space of flat surfaces of genus 2}} 
\label{sec:TMing2}

In this section we describe loci in the universal covering of
$\oFamilyoodual$ in terms of theta functions, their derivatives
and the torsion sections with the following properties. First, 
they are invariant under the covering group and hence they descend
to loci in $\oFamilyoodual$. Second, their images in the
pseudo-Hilbert modular surfaces are the \Teichmuller curves we are
interested in, or rather a union of these. 
\par
For this purpose we take for $d$ odd the unique even Hilbert theta 
function $\vartheta = \vartheta\Tchi{0,1}{1,0}$ whose characteristic 
is invariant under $\SLoodual$ (see Section~\ref{sec:HilbTheta}), 
and for $d$ even we take one of the Hilbert
theta function with even characteristic in the orbit $E_0$, say $\vartheta =
\vartheta\Tchi{0,0}{0,0}$. We let 
\be \label{eq:defU}
U: \HH^2\times\CC^2 \to \oFamilyoodual
\ee
be the universal covering map. 

\subsection{The stratum $\omoduli[2](1,1)$}

We fix a torsion order $m \in \NN$
and define $\tilde O_m(1,1)$,  the {\em lifted origami locus} for the
stratum $\omoduli[2](1,1)$. These
are points on the theta divisor, where the derivative of theta in 
the $u_2$-direction vanishes and whose first coordinate projects
to an $m$-torsion point. Formally, 
\be \label{eq:deforigamiHH}
\tilde O_m(1,1) = \Bigl\{(z,u)\in \HH^2\times \CC^2\, :\,
\vartheta(z,u) = 0,\,\,\, \pder[\vartheta]{u_2}(z,u) = 0,\,\,
(z,u)\in U^{-1}(N^{(1)}_{m\ttors}) \Bigr\}.
\ee
\par
The transformation properties of theta functions imply that
the images of the lifted origami loci are closed (in fact algebraic) 
subsets of the (open) universal families.
\par
\begin{lemma} \label{le:thder_closed}
The images $O^\circ_m(1,1) = U(\tilde{O}_m(1,1))$ for any $m \in \NN$ 
are closed subsets of  $\oFamilyoodual$.
\end{lemma}
\par
We are ultimately interested in their closures in the compactified universal family.
\par
\begin{defn}
The {\em origami locus} $O_m(1,1)$ is the
closure in $\Familyoodual$ of $O^\circ_m(1,1))$.
\end{defn}
\par
In this section, we show that the $\pi$-push forward of $O_m(1,1)$ 
is a union of
arithmetic \Teichmuller curves in $\omoduli[2](1,1)$ plus possibly some spurious 
parts of the reducible locus and of 
arithmetic \Teichmuller curves in $\omoduli[2](2)$ if $m =1,2$.
\par
\begin{theorem}\label{thm:classes-H11-tors-M}
Let $m\in \NN$, $m > 1$. If $m\congruent 0\bmod 2$, then
\[\pi_*O_{2m}(1,1) = 2T_{d, \tors=m}.\]
If $m\congruent 1\bmod 2$, then
\begin{align*}
  \pi_*O_{2m}(1,1) &= 2T_{d, \tors=m,\spin=1}, \quad & 
  \pi_*O_{m}(1,1)  &= 2T_{d, \tors=m,\spin = 3}, \quad & \text{for}\,\, d\,\,\text{odd} \\
  \pi_*O_{2m}(1,1) &= 2T_{d, \tors=m,\spin=2}, \quad & 
  \pi_*O_{m}(1,1)  &= 2T_{d, \tors=m,\spin = 0}, \quad & \text{for}\,\,d\,\,\text{even}.
 \end{align*}
\end{theorem}
\par
The case $m=1$ is special in that we also hit Teichm\"uller curves in
$\omoduli[2](2)$ and parts of the reducible locus by $\pi_*O_1(1,1)$. 
\par
\begin{theorem}
\label{thm:classes-H11-tors-zero}
 The push-forward of the origami locus decomposes as
 \begin{align*}
  \pi_*O_{1}(1,1) &= 
      2[T_{d,\tors = 1,\spin = 3}] + 3 [W_{d^2,\spin = 3}] + \phantom{3}[P_{d^2,\spin = 3}] & d\,\,\, \text{odd}\\
  \pi_*O_{2}(1,1) &= 
      2[T_{d,\tors = 1,\spin =1}] + 3 [W_{d^2,\spin = 1}] + \phantom{3} [P_{d^2,\spin = 1}] & d\,\,\, \text{odd} \\
  \pi_*O_{1}(1,1) &= 
      2[T_{d,\tors = 1,\spin = 0}]  \phantom{+ 55[W_{d^2,\spin = 1}]} + \phantom{3} [P_{d^2,\spin = 0}] & d \,\,\,\text{even}\\
  \pi_*O_{2}(1,1) &= 
      2[T_{d,\tors = 1,\spin =2}] + 3[W_{d^2,\spin = 2}] + \phantom{3} [P_{d^2,\spin = 2}] & d\,\,\, \text{even} \\
 \end{align*} 
\end{theorem}
\par
We start the proofs with the closedness lemma.
\par
\begin{proof}[Proof of Lemma~\ref{le:thder_closed}]
The vanishing locus of a Hilbert Jacobi form is closed, since it
is a closed subset of $\HH^2 \times \CC^2$ and since the
automorphy factor is a product of non-zero terms. This applies for the
full group $\Semidirectoodual$ for $d$ odd, and for a subgroup of
finite index in $\Semidirectoodual$ that stabilizes the characteristic 
(see Section~\ref{sec:modifeven}) for $d$ even. Arguing for this subgroup
is sufficient since the image of a closed set under a finite map is again closed.
\par
The torsion condition is also closed. It remains to treat the derivative of 
the theta function. We define $\chi(M,r) = \chi_\theta(\Psi(M,r))\e(\tr_{K/\QQ}(\tfrac{\tilde\gamma_1}{2}r_2 - \tfrac{\tilde\gamma_2}{2}r_1)$.
Restricted to points $(z,u)$ where $\vartheta(z,u)=0$ (and hence
also $\vartheta((M,r)(z,u))$ we obtain for all $(M,r)\in \Semidirectoodual$
by differentiating the equation defining modularity 
(see Proposition~\ref{prop:thetaisHJF})
and using the definition of the action in~\eqref{eq:semidir_act} that 
\begin{flalign} 
\pder[\vartheta]{u_2}((z,u)) &\=  \frac{\partial}{\partial u_2} \Bigl( \vartheta((M,r)(z,u)) \, \tilde\jmath_{(\tfrac12,\tfrac12),(\tfrac12,\tfrac12)}((M,r),z,u) \, \chi(M,r) \Bigr)\, & \nonumber \\ 
&\= \pder[\vartheta]{u_2}((M,r)(z,u))\bigl(c^{(2)}z_2 + e^{(2)}\bigr)^{-1} \,\, 
\tilde\jmath_{(\tfrac12,\tfrac12),(\tfrac12,\tfrac12)}((M,r),z,u) \, 
\chi(M,r)\, & \nonumber\\ 
&\=  \pder[\vartheta]{u_2}((M,r)(z,u))\,\, 
\tilde\jmath_{(\tfrac12,\tfrac32),(\tfrac12,\tfrac12)}((M,r),z,u) \, \chi(M,r)\,.
\label{eq:Mod_derivTheta}
\end{flalign}
Consequently, the automorphy factor is here again a product of non-zero 
terms and the vanishing locus is well-defined and closed as a subset of 
 $\oFamilyoodual$ for both parities of~$d$.
\end{proof}
\par
As first step towards the theorems of this section, we show that the 
origami maps are normalized in the sense of Proposition~\ref{prop:kani_norm} 
for the two theta functions we need.
\par
\begin{lemma}\label{lem:thetafunction-normalized}
 For fixed $z\in\HH^2$, let $\Theta_z\Tchi{\gamma_1}{\gamma_2}$ 
 denote the curve in $A_{d^2,z}$ given by 
 $\vartheta\Tchi{\gamma_1}{\gamma_2}  = 0$.
 The covering 
 \[\pr_i :\Theta_z\Tchi{\gamma_1}{\gamma_2} \to E_{({z_i}/{d},1)}\]
 is normalized, if and only if
 \begin{align*}
  \Tchi{\gamma_1}{\gamma_2} \in E_0, \ 
\text{if}\,\, d \congruent 0
\bmod 2 \qquad\text{resp.}\qquad
  \Tchi{\gamma_1}{\gamma_2} = \Tchi{0,1}{1,0}, \ \text{if}\,\, 
d\congruent 1 \bmod 2.
 \end{align*}
\end{lemma}
\begin{proof}
 The function $\vartheta = \vartheta\Tchi{0,0}{0,0}$ vanishes at all odd
2-torsion
 points, since translating $\vartheta$ by such a point gives a theta function
with
 odd characteristic. Since $\Theta_z =
\Theta_z\Tchi{\tilde\gamma_1}{\tilde\gamma_2}$ 
 is a symmetric divisor with respect to $[-1]$, the translates by 
 $z^*\tilde\gamma_1^T + \tilde\gamma_2^T$ of the odd 2-torsion points are
 precisely the 6 Weierstra\ss{} points on $\Theta_z$. The claim now follows by
 inspecting Table~\ref{table:oddthetachars}.
\end{proof}
\par
\begin{proof}[Proof of Theorems~\ref{thm:classes-H11-tors-M} and~\ref{thm:classes-H11-tors-zero}]
A point $z \in \oHMSoodual$ lies in the support of $\pi_*O_m(1,1)$
if and only if it has a
 preimage $y\in A_{d^2,z}$ such that $y\in \Theta_z$, such that 
 $y$ is a ramification point of $p_1:\Theta_z\to E_{z_1/d,1}$, or 
 alternatively a zero of the first eigendifferential $\omega_1 = \pi_1^* \omega_{E}$, 
and such that 
 $y$ is mapped to a $m$-torsion point in $E_{z_1/d,1}$.
 \par
If $y$ is a ramification point of order $2$,
then it is a fixed point of the hyperelliptic involution, 
so it is a Weierstra\ss{} point. Consequently $z\in W_{d^2}$
and such a point has a unique preimage in  $O_m(1,1)$.
\par
Suppose that $y$ is a ramification point of order $1$ and that $\Theta_z$ is a smooth curve.
Then  two zeros of $\omega_1$ are exchanged by the
hyperelliptic involution $\sigma$, and $\sigma$ descends to the elliptic
involution (see Proposition~\ref{prop:kani_norm}). Hence the images of the 
ramification points differ
by a torsion point on $E_{z_1/d,1}$ and $z$ lies on some  $T_{d,\tors,\spin}$.
The torsion order of the corresponding minimal covering is $m$ or $m/2$, 
depending on $m \mod 4$, on $d$ and $\spin$, as explained in Section~\ref{subsec:spin}. 
This implies the set-theoretic assignment of the various $T_{d,\tors,\spin}$
to the push-forwards of the $O_m(1,1)$. In each of the cases there are two
possible points $y$ for the same $z$.
\par
 If $\Theta_z$ is a singular curve, then it is reducible, and its components
 are two elliptic curves $E_1$, $E_2$ joined at a node, since  $\pi_*O_m(1,1)$
is the closure of a subvariety in $\oHMSoodual$ for any $m$, and hence the Jacobian of
a generic point of its support is compact. On each 
$E_i$ ($i=1,2$), the projection $p_1$ is still non-constant
(since $p_1$ and the projection to the kernel of $p_1$ deform 
over all of  $\oHMSoodual$, otherwise the splitting as product of elliptic
curves would deform to  all of  $\oHMSoodual$), and thus an 
unramified covering. Consequently, $\tfrac{\partial \vartheta}{\partial u_2}$
never vanishes at a smooth point of $\Theta_z$, while it does 
vanish at the  singular point of $\Theta_z$ (even both
partial derivatives of $\vartheta$ vanish). 
\par 
The node $y$ is a $2$-torsion point different from
the six odd Weierstra\ss{} points, i.e.\ it is an even $2$-torsion point.
Consequently, its $p_1$-image is a $2$-torsion point and there is no
contribution from the reducible locus, except for $m=1$ and $m=2$.
\par
\medskip
Suppose first that $d$ is odd, hence $\vartheta = \vartheta\Tchi{0,1}{1,0}(z)$.
If the node is mapped to zero, then it is an even two-torsion point with the
property that after translating by $\Tchi{0,1}{1,0}$ its $p_1$-image is zero, 
i.e.\ in the eigenform coordinates of the second row of the 
Table~\ref{table:eventhetachars} the first column of the point is zero.
By inspecting the table we see that there is only one possibility,
$\Tchi{0,1}{1,0}$ itself. This implies that $y=0$ and that $z$ is
in the vanishing locus of the corresponding theta constant, i.e.\ 
$\vartheta_0\Tchi{0,1}{1,0}(z) = 0$. By Lemma~\ref{le:Pd_modularform} 
this is the defining equation of $P_{d^2,\spin=3}$.  
\par
Similarly, precisely the odd theta characteristics in $O_9$ are mapped 
after translation by $\Tchi{0,1}{1,0}$ to a primitive $2$-torsion point. 
By Lemma~\ref{le:Pd_modularform} this implies that $P_{d^2,\spin=3}$
is contained in $\pi(O_2(1,1))$.
\par
\medskip
Suppose next that $d$ is even and $\vartheta = \vartheta\Tchi{0,0}{0,0}(z)$.
Precisely the odd theta characteristics in $E_0$  are mapped
(after translation by zero and base change) to a first column equal to zero, 
while those in $E_2$ are mapped to a primitive $2$-torsion point.
Together with Lemma~\ref{le:Pd_modularform} this explains the 
setwise distribution
of the reducible locus among  $\pi(O_1(1,1))$ and  $\pi(O_2(1,1))$.
\par
\medskip
It remains to determine the multiplicities of $O_m(1,1)$ at
the components lying over the curves 
$T_{d,\tors, \spin}$, $W_{d^2,\spin}$ and $P_{d^2,\spin}$.
We start with  $W_{d^2,\spin}$. Fix $q_1$, an $\tors$-torsion point $u_1$
and shift the remaining coordinates, so that in the new coordinates
the point will be at $\tilde{z_2} = 0$ and $\tilde{u_2}=0$. 
The fiber of the origami locus is cut out by 
$\tilde{\vartheta}(\tilde{z_2},\tilde{u_2}) = 0$ and  
$\partial_{u_2} \tilde{\vartheta}(\tilde{z_2},\tilde{u_2}) = 0$
for some function $\tilde{\vartheta}$, which is odd as a function of 
$\tilde{u_2}$. This implies that the multiplicity of the fiber is two, 
hence the multiplicity of the component is a multiple of two.
Now we consider the fiber with $(z_1,u_1)$ varying, choosing locally
$(z_2,u_2)$ so the the first two conditions of the origami locus
are satisfied. 
Since locally near the critical point three branches of the map $p_1$ 
come together, the multiplicity of the component is divisible by three.
Taking the factor $1/2$ from the quotient stacks into account, 
this implies that the multiplicity of $W_{d^2,\spin}$ is three.\footnote{A priori,this argument shows that the multiplicity is at least three. 
Similarly, the arguments in the subsequent paragraphs show that the 
coefficients on the right hand sides
are at least what is written in Theorem~\ref{thm:classes-H11-tors-zero}
resp.\ Theorem~\ref{thm:classes-H11-tors-M}.
Since we know the total count by an independent argument, 
see Proposition~\ref{prop:kani_EMS}, 
the multiplicities cannot be larger.}
\par
Near $T_{d,\tors,\spin}$ the branching argument for $p_1$ gives
multiplicity two. Two preimages and the stacky factor $1/2$ 
give in total the coefficient two in  
Theorem~\ref{thm:classes-H11-tors-zero}.
\par
Near $P_{d^2,\spin}$ the fiber is singular near the preimage point $z$, 
hence besides $\partial_{u_2} \vartheta$ also $\partial_{u_1} \vartheta$
vanishes there. This implies multiplicity at least two, hence
at least one, with stacky factor $1/2$ taken into account.
\end{proof}
\par

\subsection{The stratum $\omoduli[2](2)$} \label{sec:TM_2}

We need the following theorem from \cite{bainbridge07} to subtract
the contribution of the curves $W_d^\spin$ that appear in Theorem~\ref{thm:classes-H11-tors-zero}.
\par
\begin{theorem} \label{thm:intro_class2}
The classes in  $\CH^1(\HMSoodual)$ of the  \Teichmuller curves
generated by reduced square-tiled surfaces in $\omoduli[2](2)$ are
\begin{align*}
[W_{d^2}^{\spin = 3}] &\= \tfrac32(1-\tfrac3{d}) \lambda_1 + \tfrac92(1-\tfrac3{d})\lambda_2 
\quad \text{and} \quad & [W_{d^2}^{\spin = 1}] &\= \tfrac32(1-\tfrac1{d})  \lambda_1 + \tfrac92(1-\tfrac1{d})\lambda_2
\end{align*}
for $d$ odd and
\begin{align*}
[W_{d^2}] =  [W_{d^2}^{\spin = 2}] &\= 3(1-\tfrac2{d})  \lambda_1 + 9(1-\tfrac2{d})\lambda_2
\end{align*}
for $d>2$ even. Consequently, the number $w_d^\spin$ of reduced square-tiled surfaces  in $\omoduli[2](2)$
with spin $\spin$ is
\begin{align*}
w_d^{\spin =3} &\= \frac3{16}(d-3)\tfrac{\Delta_d}{d} \quad & w_d^{\spin =1} &\=\frac3{16}(d-1)\tfrac{\Delta_d}{d} \\
w_d^{\spin =2} &\=\frac3{8}(d-2)\tfrac{\Delta_d}{d}
\end{align*}
where the first line corresponds to $d$ odd and the second to $d>2$ even.
\end{theorem}
\par
The counting part of this theorem was proven in \cite{LelievreRoyer06}, 
the class in $\CH^1(\HMSoodual)$ was first determined in \cite{bainbridge07}. 
\par
We sketch how one could prove this theorem, at least without distinguishing
the components, with a similar setup as for the stratum $\omoduli[2](1,1)$.
We define $\tilde O(2)$,  the  lifted origami locus for the
stratum $\omoduli[2](2)$ to be 
\bes
\tilde O(2) = \Bigl\{ (z,u)\in \HH^2\times
\CC^2\,:\,\, \vartheta(z,u) = 0,\quad \pder[\vartheta]{u_2}(z,u) = 0,\quad
\pdern[\vartheta]{u_2}{2} = 0 \Bigr\}.
\ees
\par
The transformation properties of theta functions imply again that
$O^\circ(2) = U(\tilde{O}(2))$ is closed in $\oFamilyoodual$.
The origami locus $O(2)$ is defined as the closure in $\Familyoodual$ of $O^\circ(2))$.
With similar arguments as above one can show that the push-forward of the origami locus
$O_2$ is supported on  $W_{d^2}^{\spin}$.
To prove Theorem~\ref{thm:intro_class2} 
from here it remains to determine the multiplicity of this push-forward and
compute the class of $\pi_*O(2)$
as a triple intersection, following the proof for $\pi_*(O_m(1,1))$ given
in the next section.

\section{Intersection products} \label{sec:intprods}

We now can complete the proof of Theorem~\ref{thm:intro_class11}. 
For this purpose we prove in Theorem~\ref{thm:howtointersect} 
how to subtract from a triple intersection
of divisors on $\Familyoodual$ suitable boundary components in order to compute
the class of the pushforward of the origami locus $O_m(1,1)$. 
As technical steps it remains to actually perform triple intersection
of the geometric divisors appearing on the right hand side of the
class computation in 
Theorem~\ref{thm:class_of_HJF} 
(see Proposition~\ref{prop:triple-intersection}) and to compute these
boundary contribution.
\par
In this section, we restrict to the case $d$ odd. The additional computations
that have to be performed for even $d$ are briefly discussed in 
Section~\ref{sec:modifeven}.
We continue to denote by $\vartheta$ the unique Hilbert theta function with 
even characteristic fixed by $\Gamma_{d^2}$. It gives rise to a
section of the Hilbert-Jacobi bundle $\Jacforms_\vartheta =
\Jacforms_{\kappa,m}(\Semidirectoodual)$ with $\kappa = m =
(\sfrac12,\sfrac12)$, and therefore to a Cartier divisor $\div\vartheta$ on
$\Familyoodual$. The associated Weil divisor $[\div\vartheta]$ can be written as
\[ [\div\vartheta] = \Theta + B(\vartheta)\]
where $B(\vartheta)$ is a linear combination of boundary components 
and $\Theta$ has no support at the boundary. We view $\Theta$
as element of $\CH^1(\Familyoodual)$ .
\par
Let $|\Theta|$ denote the support of $\Theta$ and let $i:|\Theta|\to
\Familyoodual$ be the inclusion. We can compare the intersection
numbers on $\Theta$ and $\Familyoodual$ since $\Theta$ is a reduced
(and in fact irreducible) subvariety.
\par
The next condition in the definition of the origami locus is the
vanishing of the theta derivative. On $\Theta$ this function
is a section of the restriction of a bundle on $\Familyoodual$, 
whose class we already computed. Recall the definition of~$U$
from~\eqref{eq:defU} 
\par
\begin{prop}
The function $\pder[\vartheta]{u_2}$ restricted to $U^{-1}(|\Theta|)$ descends
to a well-defined global meromorphic section $\partial\vartheta$ of
$i^*\Jacforms_{\partial\vartheta}$, where $\Jacforms_{\partial\vartheta}$ is
the bundle of Hilbert Jacobi forms $\Jacforms_{\kappa,m}(\Semidirectoodual)$ with $\kappa=(\tfrac12,\tfrac32), m =
(\tfrac12,\tfrac12)$.
\end{prop}
\par
\begin{proof}
This follows immediately from the calculation in~\eqref{eq:Mod_derivTheta}.
\end{proof}
\par
To the Cartier divisor $\div\partial\vartheta$ we associate the Weil divisor
$[\div\partial\vartheta]$. It is a sum
\[ [\div\partial\vartheta] = D\Theta + B(\partial\vartheta)\]
where $B(\partial\vartheta)\in \CH^1(|\Theta|)$ is a linear combination of
boundary components of $\Theta$, and $D\Theta$ has no support on the boundary.
\par
Finally, in the definition of the origami locus, we have to intersect
with the torsion condition. This may also result in components, that
lie entirely in the boundary. We have to subtract this contribution, 
that is, in $\CH^*(|\Theta|)$, we can write
\[i^*[O_m(1,1)] = D\Theta.i^*(N^{(1)}_{m\ttors}) - B_m(N)\]
where $B_m(N)$ is supported in the boundary of $|\Theta|$ since
by definition $O_m(1,1)$ has no support on the boundary.
\par
\begin{theorem} \label{thm:howtointersect}
For $d$ odd, the class of the origami locus in $\CH^*(\Familyoodual)$ can
be computed as 
\ba \label{eq:OMdiff}
\phantom{m}
 [O_m(1,1)] &\=
  \Chern_1(\Jacforms_\vartheta).
  \Chern_1(\Jacforms_{\partial\vartheta}).
  N^{(1)}_{m\ttors} -
  B(\vartheta). \Chern_1(\Jacforms_{\partial\vartheta}).N^{(1)}_{m\ttors} \\
&\phantom{\=} -
  N^{(1)}_{m\ttors}.i_* B(\partial\vartheta) - i_*B_m(N)\, .
\ea
\end{theorem}
\begin{proof}
Since $\Theta$ is reduced, the pushforward of $D\Theta$ by $i$ is
\[i_*D\Theta = \Chern_1(\Jacforms_{\partial\vartheta}).\Theta -
i_*B(\partial\vartheta).\]
by the projection formula. Thus,
 \begin{align*}
  [O_m(1,1)] + i_*B_m(N) &= i_*(i^*N^{(1)}_{m\ttors}.D\Theta)\\
    &= N^{(1)}_{m\ttors}.\Chern_1(\Jacforms_{\partial\vartheta}).\Theta \, -\, 
N^{(1)}_{m\ttors}.i_*B(\partial\vartheta)\, .
 \end{align*}
Now plug in $\Theta = \Chern_1(\Jacforms(\vartheta)) - B(\vartheta)$ to obtain
the claim.
\end{proof}

\subsection{Triple intersections} \label{sec:tripleintersection}
The divisors of Jacobi forms have been expressed in term of the zero
section divisors $N^{(i)}$, the pullbacks of Hodge bundles $\pi^* \lambda_i$.
The evaluation of intersection products of those divisors and
with the  boundary divisors $D^{(i)}$ is manageable since
many triple intersections have $\pi$-pushforward equal to zero.

\begin{prop}\label{prop:triple-intersection}
The $\pi$-pushforward of a triple intersection between any of $N^{(i)}$,
$\pi^*\lambda_i$, and $D^{(i)}$ is given by
\begin{align*}
\pi_*(N^{(1)}.N^{(2)}.\pi^*\lambda_i) &= d^2 \lambda_i &
\pi_*(N^{(1)}.N^{(2)}.D^{(i)}) &=d^2 R^{(i)} \\
\pi_*((N^{(1)})^2. N^{(2)}) &= -d^2 \lambda_1 &  
\pi_*(N^{(1)}. (N^{(2)})^2)  &= -d^2 \lambda_2 
\end{align*}
and is zero for all triples that do not agree with any of the above up to 
permutation.
\end{prop}
\begin{proof} The divisors $\pi^*\lambda_i$ and $D^{(i)}$ are vertical, i.e.\
their $\pi$-images are divisors, while the $N^{(i)}$ are horizontal, 
i.e.\ $\pi|_{N_i}$
is surjective. Consequently, any intersection of three
divisors meeting properly, among which two are vertical, consists
of $1$-cycles along which $\pi$ is of relative dimension $\geq 1$, 
hence their $\pi$-pushforward is zero. We may use linear equivalence
in the base to ensure that the proper intersection hypothesis holds
for any of the intersections $N^{(i)}.\pi^*\lambda_j.\pi^*\lambda_k$
 for $i,j,k\in\{1,2\}$.
\par
The intersection $N^{(1)}.N^{(2)}$ is the closure of the projection of 
\[\set{(z,u)\in \HH^2\times \CC^2}{u \in
\diag(\tfrac{z_1}{d},\tfrac{z_2}{d})\ZZ^2+ \ZZ^2}\]
to $\Familyoodual$. In each fiber, this is a group of order $d^2$, the kernel of
the projection
to $\CC/(\tfrac1dz_1,1)\ZZ^2\times \CC/(\tfrac1dz_2,1)\ZZ^2$. Thus,
$\pi_*(N^{(1)}.N^{(2)}) = d^2\pi_*N = d^2[\HMSoodual]$,
where $N$ is the zero section of 
$\pi: \Familyoodual \to \HMSoodual$. This gives all the intersection products 
with $\pi$-pullbacks as stated.
\par
It remains to treat intersections of $\pi$-pullbacks 
with $(N^{(i)})^2$. Since $(N^{(i)})^2$ is
represented by the pullback via $\pr_i$ of a zero-cycle  
$\Familyelldleveld$, its intersection with any of the vertical divisors
is a cycle on which $\pi$ is of relative dimension one, hence again 
its $\pi$-pushforward is zero. 
\par
For the remaining  two cases stated in the last
line of the lemma we start with 
$\varpi_*(N_{X(d)}^2) = - \lambda_{X(d)}$, as in the proof 
of Proposition~\ref{prop:Ntor1tor}. This directly implies that
\[ (\varpi \times \varpi)_* \bigl((N_{\Box}^{(1)})^2 . N_{\Box}^{(2)}\bigr) = 
- \lambda_{\Box}^{(1)},\]
using the commutativity of the diagram
\begin{diagram}[width=6em,height=6ex]
 \CH^*_\QQ(\Familyelldleveld)\tensor \CH^*_\QQ(\Familyelldleveld) & \rTo &
\CH^*_\QQ(\Familyelldleveld^2)\\
 \dTo^{\varpi_* \tensor \varpi_*} & & \dTo_{(\varpi\times\varpi)_*}\\
 \CH^*_\QQ(\Modulielldleveld)\tensor \CH^*_\QQ(\Modulielldleveld) & \rTo &
\CH^*_\QQ(\Modulielldleveld^2)
\end{diagram}
and the fact that $N_{\Box}^{(2)}$ is the pullback of a section the second
elliptic fibration. The same argument gives 
\[(\varpi \times \varpi)_* ((\mu_* N_{\Box}^{(1)})^2 . \nu_* N_{\Box}^{(2)}) = 
- \lambda_{\Box}^{(1)}\]
for any translates by  
torsion sections $\mu$ and $\nu$. Now
\begin{align*}
\pi_*((N^{(1)})^2. N^{(2)}) &= \frac1{d^2 \Delta_d} \pi_* \tilde\tau_*(
(\tilde\tau^*N^{(1)})^2\,.\, \tilde\tau^*N^{(2)}) \\
& = \frac1{d^2 \Delta_d} \tau_* (\varpi \times \varpi)_*
\left(\sum_{\mu \in T} \sum_{\nu \in T} 
(\mu_* N_{X(d)}^{(1)})^2 \,. \,\nu_* N_{X(d)}^{(2)} \right) \\
&= \frac{d^2}{\Delta_d} \tau_* (-\lambda^{(1)}_{X(d)}) \, = \, \, - d^2
\lambda^{(1)}
\end{align*}
where $T = \oodual/(\ZZ^2\oplus d\ZZ^2)$ is a torsion subgroup of order $d^2$
and where
we used that for
$\mu, \mu' \in T$ we have
$(\mu_* N_{X(d)}^{(1)}\,.\, \mu'_* N_{X(d)}^{(1)}) = 0$  unless $\mu= \mu'$.
\end{proof}
\par

\subsection{Boundary contributions} \label{sec:bdcontrib}

In this section we collect all the boundary contributions that
appear in Theorem~\ref{thm:howtointersect}. Together with
the results from Section~\ref{sec:TMing2} this allows us  to 
conclude the proof of the main Theorem~\ref{thm:intro_class11} for $d$ odd.
The proofs of the boundary statements appear in the next section.
\par
\begin{prop} \label{prop:Btheta}
For $d$ odd the boundary contribution of $\div\vartheta$ in $\CH^1(\Familyoodual)$ is
$$ B(\vartheta) = \tfrac18 \bigl( D^{(1)} + D^{(2)}\bigr)\,.$$
\end{prop}
\par
\begin{prop} \label{prop:BDtheta}
For $d$ odd the boundary contribution of $\div\partial\vartheta$ 
in $\CH^2( \Familyoodual)$ is equal to 
\be \label{eq:BDtheta}
B(\partial\vartheta) = \tfrac18 \bigl( D^{(1)} + D^{(2)}\bigr).\Chern_1(\Jacforms_\vartheta)\,.
\ee
\end{prop}
\par
\begin{prop} \label{prop:BmN}
For $d$ odd the push-forward of the boundary contribution $B_m(N)$ is equal to 
\be \label{eq:BDN}
\pi_*B_m(N) = \begin{cases} 
	       R^{(2)} 	& \text{if} m=1\\
               0	& \text{else}
              \end{cases}
\,.
\ee
\end{prop}
\par
\medskip
\begin{proof}[Proof of Theorem ~\ref{thm:intro_class11}] 
There are several cases to be discussed.
\par
\noindent
{\em Case $M>1$ odd, spin $\spin=3$.}\ 
In this case $2[T_{d,M,\spin=3}] =  [\pi_* O_M(1,1)]$ by 
Theorem~\ref{thm:classes-H11-tors-M}. The first contribution to 
this is,  according to Proposition~\ref{prop:Ntor1tor} and 
Theorem~\ref{thm:howtointersect}, equal to
\begin{flalign}
 \label{eq:JTJDTmtor}
\pi_*( \Chern_1(\Jacforms_\vartheta).
  \Chern_1(\Jacforms_{\partial\vartheta}).
  N^{(1)}_{M\ttors}) &\= \pi_* \Bigl(\bigl((\tfrac12+\tfrac1d)\pi^* \lambda_1 + (\tfrac12 +\tfrac1d)\pi^* \lambda_2 + \tfrac1d N^{(1)} + \tfrac1d N^{(2)}\bigl). \Bigl. \nonumber\\ 
& \!\!\!\!\!\!\!\!\!\!\!\!\!\!\!\!\!\!\!\!\!\!\!\!\!\!\!\!\!\!\!\!\!
.\, \Bigl. \bigl((\tfrac12+\tfrac1d)\pi^* \lambda_1 + (\tfrac32+\tfrac1d)\pi^* \lambda_2 + \tfrac1d N^{(1)} + \tfrac1d N^{(2)}\bigr) \, \,.\,\, \tfrac{\Delta_M}{M}\bigl(N^{(1)} + \lambda_1\bigr) \, \Bigr) \\
&\= d\, \tfrac{\Delta_M}{M}\bigl( (1+\tfrac2d) \lambda_1 + (2+\tfrac1d)\lambda_2 \bigr)\,. \nonumber
\end{flalign}
Next, 
\begin{flalign} \label{eq:JTJDB}
\pi_*(B(\vartheta). \Chern_1(\Jacforms_{\partial\vartheta}).N^{(1)}_{M\ttors}) &\= \pi_* \Bigl(\bigl((\tfrac12+\tfrac1d)\pi^* \lambda_1 + (\tfrac32 +\tfrac1d)\pi^* \lambda_2 + \tfrac1d N^{(1)} + \tfrac1d N^{(2)}\bigl).  \nonumber \\
& \quad \quad \,.\, \Bigl.\tfrac{\Delta_M}{M} \bigl(N^{(1)} + \lambda_1\bigr)      \, .\, \tfrac18 (D^{(1)} + D^{(2)}) \Bigr)  \nonumber\\
&\= \tfrac18\,d\,  \tfrac{\Delta_M}{M}\,\bigl(R^{(1)} + R^{(2)}\bigr) \\
&\= d\,  \tfrac{\Delta_M}{M} \,\Bigl(\tfrac{3}{2d} \lambda_1 + \tfrac{3}{2d} \lambda_2 
\Bigr) \, . \nonumber
\end{flalign}
By Proposition~\ref{prop:BDtheta} we get
\begin{flalign}
\label{eq:NtorsBDth}
\pi_*(N^{(1)}_{M\ttors}.B(\partial \vartheta)) 
&\= \pi_*(N^{(1)}_{m\ttors} .  \Chern_1(\Jacforms_{\vartheta}). \tfrac18 (D^{(1)} + D^{(2)})) \!\!\! & \nonumber \\
&\= d\, \tfrac{\Delta_M}{M} \,\Bigl(\tfrac{3}{2d} \lambda_1 + \tfrac{3}{2d} \lambda_2 
\Bigr)\, .
\end{flalign}
Since $\pi_*(B_M(N)) = 0$ for $M>1$ we find altogether
\bas\ 
 [\pi_* O_{M}(1,1)] \=  \,d\, \tfrac{\Delta_\tors}{\tors} \,\Bigl((1 - \tfrac{1}{d})\lambda_1 + (2 - \tfrac{2}{d})\lambda_2\Bigr),
\eas
and this completes the first case.
\par
\medskip
\noindent
{\em Case $M>1$ odd,  spin $\spin=1$.}\  Since in this case
$2\, [T_{d,M,\spin=1}] =  [\pi_* O_{2M}(1,1)]$ and since $N_{2M\ttors}^{(1)} = 
\tfrac{3\Delta_M}{M}N^{(1)}$ all the contributions are multiplied by three
compared to the previous calculation, and this proves the second case.
\par
\medskip
\noindent
{\em Case $M$ even.}\ 
Recall that there is no spin distinction in this case.
Now $2[T_{d,M,\spin=0}] =  [\pi_* O_{2M}(1,1)]$ and for $M$ even
the number of primitive $2M$-torsion points is $4\tfrac{\Delta_M}{M}$.
Hence all the contributions are $4$ times larger than in the corresponding
cases for $M$ odd and spin $\spin = 3$, completing
the discussion of this case.
\par
\medskip
\noindent
It remains to discuss the subcases for $M=1$.
\par
\medskip
\noindent
{\em Case $M=1$, spin $\spin =3$.}\ 
 We compute as in~\eqref{eq:JTJDTmtor}, \eqref{eq:JTJDB} 
and~\eqref{eq:NtorsBDth}, taking into account that $N^{(1)}_{1\ttors}$
has {\em no} $\lambda_1$-contribution (as $N^{(1)}_{m\ttors}$ had it according
to Proposition~\ref{prop:Ntor1tor}),
\bas \pi_*( \Chern_1(\Jacforms_\vartheta).
  \Chern_1(\Jacforms_{\partial\vartheta}).
  N^{(1)}_{1\ttors}) &\= d\, \bigl( \lambda_1 + (2+\tfrac1d)\lambda_2 \bigr)\, , \\
\pi_*(  \Chern_1(\Jacforms_{\partial\vartheta}).N^{(1)}_{1\ttors}.B(\vartheta))
&\= d\,  \,\Bigl(\tfrac{3}{2d} \lambda_1 + \tfrac{3}{2d} \lambda_2 \Bigr) \\
\pi_*(N^{(1)}_{1\ttors}.B(\partial \vartheta))
&\= d\,  \,\Bigl(\tfrac{3}{2d} \lambda_1 + \tfrac{3}{2d} \lambda_2 
\Bigr)\, .
\eas
Since $\pi_*(i_* B_1(N)) =  R^{(2)} = \frac{12}{d}\lambda_2$ we find
\bas
\pi_* O_1(1,1) = (d-3)\lambda_1 + \tfrac2d (d^2-d-6)\lambda_2.
\eas
Subtracting the contributions from the reducible locus 
(see Proposition~\ref{prop:class_red}) and
from $W_{d,\spin=3}$ (see Theorem~\ref{thm:intro_class2}) according to 
Theorem~\ref{thm:classes-H11-tors-zero} gives the claim. 
\par
\medskip
\noindent
{\em Case $M=1$, spin $\spin = 1$.}\ Since 
$N_{2\ttors}^{(1)} = 3 (N^{(1)} + \lambda_1)$ and 
since in this case $\pi_*(B_2) = 0$ we get as in~\eqref{eq:JTJDTmtor}, 
\eqref{eq:JTJDB} and~\eqref{eq:NtorsBDth}, that
\bas
\pi_*(O_2(1,1)) = (3d -3)\lambda_1 + 6(d-1)\lambda_2.
\eas
Again, subtracting the contributions from the reducible locus 
(see Proposition~\ref{prop:class_red}) and
from $W_{d,\spin=1}$ (see Theorem~\ref{thm:intro_class2}) according to 
Theorem~\ref{thm:classes-H11-tors-zero} gives the claim. 
\end{proof}

\subsection{Intersection with the boundary: proofs }

We will deduce Proposition~\ref{prop:Btheta} from the following result.
We compute the vanishing order of the theta function for general $k$ and general characteristics, and later specialize to the unique theta function invariant under the whole group $\Semidirectoodual$.
\par
\begin{prop} \label{prop:vanordertheta}
The vanishing order at the boundary divisor $D^{(i)}_{\infty,k}$ of the theta function  $\vartheta\Tchi{\tilde\gamma_1}{\tilde\gamma_2}$ considered as a \emph{function} on the infinite chain of rational lines 
 \[\begin{cases}
    \tfrac18 - \tfrac12 k^2 , &\text{if}\quad d\tilde\gamma_1^{(i)} \congruent 1 \bmod 2\\
    -\tfrac12 k^2      , &\text{if}\quad d\tilde\gamma_1^{(i)} \congruent 0 \bmod 2
   \end{cases}
   \]
\end{prop}
\begin{proof}
By symmetry, we may focus on the case $i=1$ and compute the vanishing order of
$\vartheta\Tchi{\tilde\gamma_1}{\tilde\gamma_2}$ as \textit{a function} at
$q_{1,k}=0$. In the second line, we use the substitution $s = dx$, so that
the summation is over all $s\in \ZZ^2$ with $s'\congruent -s'' \bmod d$.
We let $\eta_i = \tfrac12 \tilde\gamma_i$. 
\begin{align*}
 \vartheta\Tchi{\tilde\gamma_1}{\tilde\gamma_2}(z,u) 
 &= \sum_{x\in \frakodual} 
      \e\bigl(\tfrac12\tr_{K/\QQ}((x+\eta_1)^2z +
2(x+\eta_1)(u+\eta_2))\bigr)\\
 &= \sum_{s'\in \ZZ} \e\bigl(\tfrac12(s'+d\eta_1')^2\tfrac{z_1}{d^2} 
			    + (s'+d\eta_1')(\tfrac{u_1}{d} +
\tfrac{\eta_2'}{d})\bigr)\\
 &\quad \cdot \sum_{s''\congruent -s'(d)}
\e(\bigl(\tfrac12(s''+d\eta_1'')^2\tfrac{z_2}{d^2} 
			    + (s''+d\eta_1'')(\tfrac{u_2}{d} +
\tfrac{\eta_2''}{d})\bigr)\\
 &= \e(\tr_{K/\QQ}(\eta_1\eta_2))
  \sum_{s'\in \ZZ} 
	  q_1^{1/2\cdot (s'+d\eta_1')^2} 
	  \zeta_1^{s'+d\eta_1'} 
	  \e(s'\tfrac{\eta_2'}{d})\\
 & \quad \cdot \sum_{s''\congruent -s'(d)} 
	  q_2^{1/2(s'' + d\eta_1'')^2}
	  \zeta_2^{s''+d\eta_1''}
	  \e(s''\tfrac{\eta_2''}{d})\\
 &= \e(\tr_{K/\QQ}(\eta_1\eta_2))
  \sum_{s'\in \ZZ} 
	  q_{1,k}^{1/2\cdot (s'+d\eta_1')^2 + k(s'+d\eta_1')}
	  \zeta_{1,k}^{1/2\cdot (s'+d\eta_1')^2 +
(k+1)(s'+d\eta_1')} \\
 & \quad \cdot 
	  \e(s'\tfrac{\eta_2'}{d})\, \sum_{s''\congruent -s'(d)} 
	  q_2^{1/2(s'' + d\eta_1'')^2}
	  \zeta_2^{s''+d\eta_1''}
	  \e(s''\tfrac{\eta_2''}{d})
\end{align*}
Note that $d\eta_1' \in \tfrac12\ZZ$. We let $\varepsilon({d\eta_1'}) = 1$, 
if $d\eta_1'$ is half-integral and $0$ if it is integral. In this
notation, the smallest $q_{1,k}$-exponent appearing in the
development of $\vartheta \Tchi{\tilde\gamma_1}{\tilde\gamma_2}$ is given by
\begin{align*}
 &\min\set{\tfrac12 (s+d\eta_1')^2 + k(s+d\eta_1')}{s\in \ZZ}\\
 &=\tfrac12 \min_{s\in \ZZ}\bigl[ s^2 + 2s(d\eta_1' + k) +
d\eta_1'(2k+d\eta_1')\bigr]\\
 &=\tfrac12 \min_{s\in \ZZ}\Bigl[\bigl(s + d\eta_1' + k\bigr)^2\Bigr] -
\tfrac12(d\eta_1' + k)^2 + \tfrac12
d\eta_1'(2k+d\eta_1')\\
 &=\tfrac18\varepsilon({d\eta_1'}) - \tfrac12 k^2\, .
 \end{align*}
This implies the claim, once we 
have checked that the corresponding coefficient is indeed non-zero. 
We may restrict to the chart $k=0$. 
If $\varepsilon(d\eta_1') = 0$, then
the minimum is attained only once for $s=-d\eta'_1$ and the coefficient
is a power of $\zeta_{1,0}$-power times a non-zero power series in $q_2$
and $\zeta_2$. This coefficient does not vanish for generic 
$(\zeta_{1,0},q_2,\zeta_2)$. 
If $\varepsilon(d\eta_1') = 1$, then the minimum is 
attained twice, for $s' +d\eta_1' = \pm \tfrac12$.
The coefficient  is of the form
\[  \zeta_{1,k}^{1/8 -k^2/2 -k + 1/2} A_1(q_2,\zeta_2) + \zeta_{1,k}^{1/8 -k^2/2 -k
- 1/2}  A_1(q_2,\zeta_2) \]
for non-zero power series $A_1$ and $A_2$. This coefficient does 
not vanish for generic $(\zeta_{1,0},q_2,\zeta_2)$ either.
\end{proof}
\par
\begin{proof}[Proof of Proposition~\ref{prop:Btheta}] 
By Lemma~\ref{lem:localcoords_at_boundaries_of_A} and
\eqref{eqn:HJF_bundel}, we can determine the vanishing order of a Hilbert
Jacobi form near $D^{(i)}$ by its Fourier development in the coordinates
$(\zeta_{1,k}, q_{1,k}, z_2, u_2)$, resp. $(z_1,u_1,\zeta_{2,k}, q_{2,k})$, and then
compare to the definition of local sections of Hilbert Jacobi forms in~\eqref{eqn:HJF_bundel}. 
Using this and plugging in the characteristic $\Tchi{\tilde\gamma_1}{\tilde\gamma_2}$ invariant under $\Gamma_{d^2}$ in the previous proposition yields the claim.
\end{proof}

For the proof of Proposition~\ref{prop:BDtheta}, let again $\vartheta$ denote the unique theta function invariant under $\Semidirectoodual$. We develop $\vartheta$ and $\partial_2\vartheta$
with respect to the boundaries. 
To this end, we introduce for $i\in \ZZ$ the functions
\begin{align}\label{eqn:theta1_i-odd}
\theta_{1,[i]} &= \sum_{s'\congruent -i (d)}
q_1^{\tfrac12 (s'-\tfrac12)^2}\zeta_1^{s'-\tfrac12} 
	\,\,	  \e(\tfrac{1}{2d}(s'+i))\\
\theta_{2,[i]} &= \sum_{s''\congruent -i (d)} q_2^{\tfrac12 {(s''+\tfrac12)^2}}
\zeta_2^{s''+\tfrac12} 
		  \e(\tfrac{1}{2d}(s''+i)) \cdot \e(\tfrac{2i-1}{4d}))
\end{align}
With the above notation, we expand $\vartheta$ and its derivative near a
divisor $D^{(1)}_{\infty,k}$ lying over the
first boundary $D^{(1)}$ as
\begin{align}
\begin{split}\label{eqn:theta-dev-1ndboundary-odd} 
\vartheta &= q_{1,k}^{\tfrac18 - \tfrac{k^2}2} \zeta_{1,k}^{\tfrac18 -
\tfrac{(k+1)^2}2} \cdot
    \Bigl(\phantom{\partial_{u_2}} \theta_{2,[-k]} +  \phantom{\partial_{u_2}} \theta_{2,[-k+1]} \zeta_{1,k} + O(q_{1,k})\Bigr)\\
 \partial_2\vartheta &=  q_{1,k}^{\tfrac18 - \tfrac{k^2}2} \zeta_{1,k}^{\tfrac18 -
\tfrac{(k+1)^2}2} \cdot
 \Bigl( \partial_{u_2}\theta_{2,[-k]} + \partial_{u_2}\theta_{2,[-k+1]}  \zeta_{1,k} + O(q_{1,k})\Bigr)\, .\\
\end{split}
\end{align}
and near a divisor $D^{(2)}_{\infty,k}$ lying over the second boundary $D^{(2)}$
\begin{align}
\begin{split}\label{eqn:theta-dev-2ndboundary-odd} 
\vartheta &\= q_{2,k}^{\tfrac18 - \tfrac{k^2}2} \zeta_{2,k}^{\tfrac18 -
\tfrac{(k+1)^2}2} \cdot
    \Bigl( \theta_{1,[-k-1]} + \zeta_{2,k} \theta_{1,[-k]} + O(q_{2,k})\Bigr)\\
\partial_2\vartheta &\= q_{2,k}^{\tfrac18 - \tfrac{k^2}2} \zeta_{2,k}^{\tfrac18 -
\tfrac{(k+1)^2}2} \cdot \Bigl(-\tfrac{1+2k}2  \theta_{1,[-k-1]} 
+ \tfrac{1-2k}2 \zeta_{2,k} \theta_{1,[-k]} + O(q_{2,k})\Bigr) \, .
\end{split}
\end{align}
\par
\begin{proof}[Proof of Proposition~\ref{prop:BDtheta}]
We have to determine the boundary contribution of $\partial_2\vartheta$ on 
$\Theta$, which is locally (using the chart $k=0$ and 
Proposition~\ref{prop:vanordertheta}) given as the vanishing locus of 
$\vartheta\,/\,q_{j,0}^{\sfrac18}\zeta_{j,0}^{\sfrac18}$ for $j=1,2$.
The factor $q_{j,0}^{\sfrac18}\, \zeta_{j,0}^{\sfrac18}$ gives, for both boundaries, 
the contribution claimed in \eqref{eq:BDtheta}. So we have to argue that
the constant terms (in $q_{j,0}$) of the remaining factors of $\vartheta$
and $\partial \vartheta$ have no common factors. Since these terms are linear
in $\zeta_{j,0}$, this holds if and only if
\bes
{\rm det}_1 \= \Bigl|\begin{matrix} \theta_{2,[0]} & \theta_{2,[1]} \\ 
\partial_{u_2}\theta_{2,[0]} & \partial_{u_2}\theta_{2,[1]} \\ \end{matrix}  \Bigr| \, \neq 0 
\quad \text{and} \quad 
{\rm det}_2 \= \Bigl|\begin{matrix} \theta_{1,[0]} & \theta_{1,[1]} \\ 
-\tfrac{1}2  \theta_{1,[0]} &  \tfrac{1}2 \theta_{1,[1]} \\ \end{matrix}\Bigr| \, \neq 0.
\ees  
Since 
\bas\
\theta_{2,[0]} &\= \e\bigl(\tfrac{-1}{4d}\bigr) \,\,\, q_2^{\tfrac18} \zeta_2^{\tfrac12} + O\bigl(q_2^{\tfrac18 {(2d-1)^2}}\bigr) \qquad\qquad \text{and} \\
\theta_{2,[1]} &\= \e\bigl(\tfrac{2d+1}{4d}\bigr)\,  q_2^{\tfrac18 {(2d-1)^2}} \zeta_2^{\tfrac12 (-2d-1)} + 
O\bigl(q_2^{\tfrac18 {(2d+3)^2}}\bigr) 
\eas
the claim for ${\rm det}_1$ is easily checked using the beginning of 
the $q_2$-expansion and for ${\rm det}_2$ the claim follows similarly.
\end{proof}
\par
\begin{proof}[Proof of Proposition~\ref{prop:BmN}]
Suppose that $(z,u)\in \HH^2\times\CC^2$ projects to $N^{(1)}_{\tors\ttors}$ 
under the universal
covering map $U$.
This is the case iff $u_1 = \tfrac{t_1}{d}z_1 + t_2$ for some
$t_1,t_2\in \tfrac1\tors\ZZ$ but there is no way to represent the point
with $t_1,t_2\in \tfrac1k\ZZ$ for any $k$ strictly dividing $\tors$.
Such a point is mapped to
$(\zeta_1,q_1) = (q_1^{t_1}\e(\tfrac{t_2}{d}), q_1).$
\par
Near the boundary $D^{(1)}$ we inspect the expansion 
\eqref{eqn:theta-dev-1ndboundary-odd} with this specialization.
Bearing in mind that $\zeta_{2,0} \neq 0$, already to first order in $q_{1,0}$
the only solution is $q_{2,0}=0$. Such a component vanishes under $\pi_*$, as claimed.
\par
Near the boundary $D^{(2)}$ we inspect the expansion \eqref{eqn:theta-dev-2ndboundary-odd}. With the substitution $r' = s'-1$ we find
\begin{align*}
\theta_{1,[-1]}(q_1^{t_1}\e(\tfrac{t_2}{d}), q_1) &= \sum_{r'\congruent 0(d)} q_1^{\sfrac12(r' - \sfrac12)^2 +
t_1(-r' + \sfrac12)}
			\e(\tfrac{t_2}{d}(-r'+\tfrac12)) (-1)^{\sfrac{r'}{d}}\\ 
		 &= \sum_{r'\congruent 0(d)} q_1^{\sfrac12(r' - \sfrac12)^2 +
t_1(-r' + \sfrac12)}
			\e(\tfrac{t_2}{2d}) (-1)^{\sfrac{r'}{d}}
\end{align*}
For $t_1=t_2=0$ this expression is equal to $\theta_{1,[0]}(q_1^{t_1}\e(\tfrac{t_2}{d}), q_1)$, hence $\det_2$ vanishes at $(q_1^{t_1}\e(\tfrac{t_2}{d}), q_1)$. 
One checks that the next term in the expansion 
(corresponding to $q_{1,0}^1$, since
$q_{1,0}^{\sfrac18}$ has been taken out) is non-zero, 
so that the multiplicity of this contribution is one, as claimed.
Hence this point $t_1=t_2=0$ contributes a divisor to $B_1(N)$, whose
$\pi$-pushforward equals $R^{(2)}$. 
\par
The substitution works for no other pair $(t_1,t_2)$. In fact, one checks
that $\det_1(q_1^{t_1}\e(\tfrac{t_2}{d}), q_1)$ has non-trivial 
$q_{1}$-expansion for any non-zero $(t_1,t_2)$. This proves
the claim.
\end{proof}

\par

\subsection{Modifications for $d$ even} \label{sec:modifeven}

Let $d>2$ be even. In this case none of the even theta characteristics in $E_0$ is
fixed by $\SLoodual$. The vanishing locus of the product
is a well-defined subvariety of $\oFamilyoodual$, but using
this product in the definition of the origami locus 
in~\eqref{eq:deforigamiHH} does not quite work since when 
taking partial derivatives, the product rule introduces a lot
of spurious components.
\par
Consequently, one has to work here with the subgroup $\Gamma'_{d^2}$ of 
$\SLoodual$ fixing the characteristic $\Tchi{0,0}{0,0}$. In fact, the
subgroup $\Gamma'_{d^2} = \diag((d,-d),1) \cdot \Gamma' \cdot
\diag((\tfrac1d,-\tfrac1d),1)$ where 
\[\Gamma' = \set{A\in \SL_2(\ZZ)^2}{A'\congruent A'' \congruent I \bmod 2,\
A'\congruent A'' \bmod 2d} 
\subset \SL_2(\frako)\,.\] 
of index $48$ has this property. Again one can compactify the open family $\HH^2\times\CC^2/\Semidirectoodual'$, where $\Semidirectoodual' = \Gamma_{d^2}' \semidirect (\frako_{d^2}^\dual \oplus \frako_{d^2})$ by employing a toroidal compactification for a normal subgroup -- in this case $\tilde\Gamma'_d(2d) = \bigl(\Gamma(2d)_d \semidirect (\ZZ\oplus d\ZZ)
\bigr)^2$ will do the job. Unfortunately, the resulting morphism $\Familyoodual' \to \Familyoodual$ from this new compactification $\Familyoodual'$ is not flat at the boundary -- it maps a folded $2$-gon to a folded $1$-gon by contracting one of the curves. One thus cannot simply pull back the relations obtained in $\Pic_\QQ(\Familyoodual)$. Instead one has to rederive the formula for the class of a Hilbert-Jacobi form (Theorem~\ref{thm:class_of_HJF}), of the section of primitive $\ell$-torsion points (Proposition~\ref{prop:Ntor1tor}), and compute the vanishing orders of the theta-function and its derivative (Section~\ref{sec:bdcontrib}).

\subsection{Intersection products and Euler characteristics}

We first convert Theorem~\ref{thm:intro_class11} into a
statement about Euler characteristics.
\par
\begin{cor}
The Euler characteristics of the arithmetic \Teichmuller curve $T_{d,\tors,\spin}$ 
are as follows. If $M>1$ is odd, then 
\ba \ 
 \chi(T_{d,\tors,\spin=3}) &\= -\frac{1}{144}(d-1)\Delta_d\frac{\Delta_\tors}{\tors}, \\
 \chi(T_{d,\tors,\spin=1}) &\= -\frac{1}{48}(d-1)\Delta_d\frac{\Delta_\tors}{\tors} \\
\ea
If $M$ is even, then $\chi(T_{d,\tors}) = -\tfrac16(d-1)\Delta_d\tfrac{\Delta_\tors}{\tors}$.
\par
\noindent
If $M=1$, then 
\ba \ 
 \chi(T_{d,\tors,\spin=3}) &\= -\frac{1}{144}(d-3)(d-5)\frac{\Delta_d}{d}, \\
 \chi(T_{d,\tors,\spin=1}) &\= -\frac{1}{48}(d-1)(d-3)\frac{\Delta_d}{d}. \\
\ea
\end{cor}
\par
\begin{proof} Pairing with $\omega_1$ and integration, as in 
Corollary~\ref{cor:eulerchar-reducible}.
\end{proof}
\par
Now we complete easily the proof of the counting theorem.
\par
\begin{proof}[Proof of Theorem~\ref{thm:intro_count11}]
Since $\chi(\HH/\Gamma(1)) = -1/6$, the number of squares is minus six times
the Euler characteristic. (This also holds if the curve is reducible.)
\end{proof}
\par
For comparison we include the proof how to deduce the total count (i.e.\
without separating the spin components) from two results in the literature.
\par
\begin{prop}[{\cite[Theorem 3]{Ka06}, \cite{ems}}] \label{prop:kani_EMS}
 The number of minimal degree $d$ covers of an elliptic curve $E'$
 branched over the divisor $P+Q$ is
 \begin{align}
 \begin{split}
  \frac13(d-1)\Delta_d, \qquad &\text{if}\ P \neq Q\\
  \biggl(\frac16(d-1) - \frac{1}{24}\frac1d(7d-6)\biggr)\Delta_d, \qquad
&\text{if}\ P=Q
  \end{split}
 \end{align}
\end{prop}
\par
\begin{cor}
 The number of square-tiled surfaces in $\omoduli[2](1,1)$ of
 degree $d$ and torsion order $\tors \geq 2$ is given by
 \[\frac13(d-1)\Delta_d\frac{1}{2\tors}\Delta_\tors.\]
\end{cor}
\begin{proof}
 Each such surface arises as a composition
 of an isogeny of degree $\tors$ with a minimal
 cover with reduced branching divisor $P + Q$.
 There are four choices to normalize it
 in such a way that $P+Q$ becomes symmetric;
 they correspond to the choice of a square-root
 of $P-Q$. After normalization, 
 $[2]P$ is of order $\tors$. Choose a
 basis of $H_1(E',\ZZ)$ in order to make an identification
 with $\ZZ^2$. Thus the $\tors$-torsion points of $E'$
 are identified with $(\ZZ/\tors\ZZ)^2$.
 Since $\SL_2(\ZZ/\tors\ZZ)$
 acts transitively on points of order $\tors$
 in $(\ZZ/\tors\ZZ)^2$, and the
 stabilizer of one of these is of order $\tors$, there
 are $\Delta_{\tors}\tfrac{1}{\tors}$ points of order $\tors$ on $E'$.
 There are $4$ choices of a square-root of $[2]P$, but
 since $P$ is determined by the covering only up to
 sign, this gives in total
 \[\frac14 \cdot \frac12 \cdot 4 \cdot \frac{1}{\tors}\Delta_{\tors} \cdot
\frac13(d-1)\Delta_d\]
 square-tiled surfaces of degree $d$ and torsion order $\tors$.
\end{proof}
\par

\section{Notations} \label{sec:notations}
We summarize the notation used for pseudo-Hilbert modular surfaces, 
the universal families over these surfaces and their coverings.
\begin{align*}
 K &\= \QQ\oplus \QQ&\\
 \frako_{d^2} &\= \{x = (x',x'') \in \ZZ\oplus \ZZ: x' \equiv x'' \bmod d \} \subset K \\
\end{align*}
Modular groups and pseudo-Hilbert modular groups
\begin{align*}
 \Gamma(\ell)   &\= \ker(\SL_2(\ZZ)\to \SL_2(\ZZ/(\ell))\quad \text{with}\ \ell\in \NN &\\ 
 \Gamma^1(d)	&\= \set{A\in \SL_2(\ZZ)}{A\congruent\tmatrix{1}{0}{\ast}{1} \bmod d}&\\
 \Gamma^1(d)^\pm &\= \Gamma^1(d)\cup \tmatrix{-1}{0}{0}{-1}\Gamma^1(d) &\\
 \Gamma(\ell)_d &\= \diag(d,1)\cdot\Gamma(\ell)\cdot\diag(d^{-1},1) & \\
\end{align*}
%
%
Semidirect products
\begin{align*}
 \tilde\Gamma(\ell)_d &\= \diag(d,1)\cdot(\Gamma(\ell)\semidirect \ell\ZZ^2)\cdot\diag(d^{-1},1) & \\
 \Semidirectoodual &\= \SL(\frako_{d^2}\oplus\frako_{d^2}^\dual)
			\semidirect (\frako_{d^2}^\dual \oplus \frako_{d^2}) &\\
\end{align*}
Open modular varieties
\begin{align*}
X(d)^\circ &\= \HH / \Gamma(d) \quad \quad \quad \quad \text{the open modular curve
with level-$d$-structure} \\
\oModulielldleveld &\= \HH / \Gamma(d)_d\ \quad \quad \quad \text{isomorphic to 
$X(d)^\circ$, uniformizing group conjugated} \\
 \oHMSoodual &\= \HH^2/  \SLoodual \quad  \quad \quad \quad \text{the open pseudo-Hilbert modular surface} \\
\end{align*}
Open universal families
\begin{align*}
\oFamilyelldleveld
&\= \HH \times \CC /  \tilde\Gamma(d)_d \quad \quad  \quad \quad \quad \text{universal
family of elliptic curves over}\,\,  \oModulielldleveld \\
\oFamilyoodual &\= \HH^2\times \CC^2/\Semidirectoodual \quad  \quad \quad \quad \text{universal
family of abelian surfaces over}\,\,  \oHMSoodual \\
\intertext{Their compactifications are denoted by the same letter without $\ ^\circ$}
\end{align*}
\par

\bibliographystyle{halpha}
\bibliography{bib_irred}
\end{document}